\documentclass[11pt,a4paper,reqno]{amsart}

\usepackage{amsmath,amsfonts,amssymb,amsthm}
\usepackage{enumitem}
\usepackage[all]{xy}
\usepackage{graphicx}
\usepackage{aliascnt}
\usepackage{hyperref}
\usepackage{geometry}
\usepackage{verbatim}
\usepackage{xcolor}
\hypersetup{
    colorlinks,
    linkcolor={blue!60!black},
    citecolor={green!80!black},
    urlcolor={blue!80!black}
}
\swapnumbers

\hypersetup{backref=true}
 
\theoremstyle{plain}
\newtheorem{theorem}{Theorem}
\newtheorem{proposition}[theorem]{Proposition}

\newtheorem{lemma}[theorem]{Lemma}
\newtheorem{corollary}[theorem]{Corollary}

\theoremstyle{definition}
\newtheorem{definition}[theorem]{Definition}

\newtheorem{example}[theorem]{Example}
\newtheorem{remark}[theorem]{Remark}

\def\Z{\mathbb{Z}}
\def\C{\mathbb{C}}
\def\N{\mathbb{N}}

\def\PP{\mathbb{P}}

\def\II{\mathrm{I}}
\def\JJ{\mathrm{J}}
\def\BB{\mathcal{B}}
\def\OO{\mathcal{O}}
\def\HH{\mathcal{H}}
\def\KK{\mathrm{K}}
\def\XX{\mathcal{X}}
\def\F1{\mathcal{F}}
\def\I1{\mathcal{I}}
\def\EE{\mathcal{E}}

\def\ker{\operatorname{ker}}
\def\im{\operatorname{im}}
\def\rank{\operatorname{rank}}
\def\dim{\operatorname{dim}}

\def\rank{\operatorname{rank}}
\def\cod{\operatorname{codim}}

\def\depth{\operatorname{depth}}

\def\Hom{{\operatorname {Hom}}}
\def\Ext{{\operatorname {Ext}}}
\def\Ex{{\operatorname {Ex}}}
\def\Tor{{\operatorname {Tor}}}

\def\Hilb{{\operatorname {Hilb}}}
\def\Quot{{\operatorname {Quot}}}
\def\min{{\operatorname {min}}}

\def\im{\text{im}\,}

\numberwithin{equation}{section}

\begin{document}

\title{Deformations of exterior differential ideals and applications.}

\author {Fernando Cukierman}
\author {C\'esar Massri}

\begin{abstract}
We develop some basic facts on deformations of exterior differential ideals on a smooth complex algebraic variety. 
With these tools we study deformations of several types of differential ideals, leading to several irreducible components of the corresponding moduli spaces.\end{abstract}



\maketitle

\tableofcontents

\newpage

\section*{Introduction.}

\

Let $X$ be a smooth algebraic variety over the complex numbers. A sheaf of differential ideals $\II\subseteq\Omega_X$ is a 
graded sheaf of coherent ideals of the exterior algebra $\Omega_X$ such that $d\II\subseteq \II$. We  shall study infinitesimal
deformations of $\II$ as a differential graded ideal of $\Omega_X$. This will follow the well developed
deformation theory; see for example \cite{MR2583634}, \cite{sernesi2007deformations}, \cite{MR2222646}, \cite{toen2017problemes}. 

\

One point  to take into account is the slight non-commutativity of $\Omega_X$. There exists a general theory of Hilbert schemes over
non-commutative rings in \cite{MR1863391} which amply covers the case of an anti-commutative sheaf of rings like $\Omega_X$.
Another less trivial point relevant for us is the condition on the ideals of being
stable under exterior derivative, that is, the integrability condition.  Our results will be formally the same as the known tangent-obstruction theory for Hilbert schemes (see e. g. \cite{MR2222646}, \cite{MR1863391}), but working in the Herrera-Lieberman category of differential complexes \cite{MR0310287}, \cite{MR1181207}, which takes  the integrability condition into account.

\

Of special relevance will be the complex $\Omega_X/\II$ which may be thought of as the De Rham complex of the super-commutative subscheme defined by the ideal
$\II \subset \Omega_X$. This or similar complexes appeared before, for instance, in \cite{gomez1988transverse} as the \emph{leaf complex}, and in \cite{bryant1995characteristic} in connection with the \emph{characteristic cohomology}.

\

With these tools at our disposal we are able to study the deformation theory of differential ideals of various special types. This allows the construction of new irreducible components of moduli spaces of differential ideals, and of singular foliations.

\

Let us summarize the contents of the different sections.  In Section 1 we collect several basic definitions and preliminary facts on differential graded modules over $\Omega_X$. In Section 2 we work out the basic theory of infinitesimal deformations and obstructions for differential graded ideals. In Section 3 we develop an alternative set-up for working with differential graded ideals on projective varieties, generalizing the Serre correspondence between coherent sheaves and graded modules. 
In Sections 4, 5 and 6 we review several known facts and we give some new technical calculations relevant to the geometry of moduli spaces of exterior differential ideals. In Section 5 we prove Theorem \ref{ext=0} on the vanishing of $\mathcal{E}xt$ for coherent sheaves, which seems to be new.
Section 7 contains our first application of the previous theory, Theorem (A)  \ref{theorem pfaff},  which gives sufficient conditions for the stability of some types of singular Pfaff ideals.  More precisely, we consider singular Pfaff ideals $\II = \langle \omega_1, \dots, \omega_q\rangle$ generated by twisted one-forms $\omega_i$ that vary in respective irreducible components $\mathcal{F}_i$ of the space of integrable one-forms. More generally, we study the stability of singular Pfaff ideals of the form $\II = \II_1 + \dots + \II_q$ where each $\II_i$ is a singular Pfaff ideal, general element of a component $\mathcal{F}_i$ of the space of singular Pfaff ideals. 

\

Theorem (A)  \ref{theorem pfaff}  is of a general nature, and in order to apply it in examples one needs to verify the surjectivity stated in hypothesis b). This surjectivity follows if the two obstruction spaces defined in (7.2) are zero. In (7.3) and (7.4) we prove that  these two spaces are actually zero, under certain more manageable hypothesis. The result obtained is stated in Theorem (B)  \ref{theorem pfaff2}. In Section 8 we discuss several  applications.

\newpage

\section{Preliminaries.} \label{preliminaries}

\
 
\subsection{Exterior differential modules.} 

\

\noindent
Let $X$ be an $n$-dimensional smooth algebraic variety over $\C$ and let $\Omega^1_X$ be
the sheaf of K\"ahler differentials. Let $\Omega_X$ be the De Rham complex of
sheaves 
\[
\Omega_X\colon\quad\OO_X\rightarrow \Omega^1_X\rightarrow\ldots\to \Omega^n_X,
\]
where the map is the exterior derivative $d$ and $\Omega^i_X=\bigwedge^i\Omega^1_X$.
In this section we recall some results about
the category of differential graded modules over $\Omega_X$. We refer to
\cite{MR0310287,MR2320462,MR3386225} and especially \cite{MR1181207} for the notation that we adopt. 
For general definitions and results
on differential graded algebras, see for example \cite{MR1258406}.

\

\noindent
Let $\Omega_X$-dgmod be the category 
of left $DG$-Modules over $\Omega_X$,
where each component of a module is 
coherent over $\OO_X$ and only a finite number
of components are non-zero, see \cite{MR1181207}.
Specifically, an object in $\Omega_X$-dgmod
${F}=\oplus_{i\in\mathbb{Z}}{F}^i$
is a collection of coherent $\OO_X$-Modules 
${F}^i$, $i\in\Z$, with a graded structure 
\[
\cdot:\Omega^k_X\otimes_{\OO_X}{F}^{i}\to {F}^{i+k}
\]
and an additive endomorphism $d:{F}\to{F}$ of degree $1$
such that $d^2=0$ and
\[
d(am)=da\cdot m+(-1)^{k}a\cdot dm,\quad
\forall a\in\Omega^k_X,\,m\in {F}. 
\]
A homomorphism 
$f:{F}\rightarrow{G}$ 
in $\Omega_X$-dgmod is a collection of $\OO_X$-linear
maps $f_i:{F}^i\to{G}^i$ such that 
$f_{i+1}\circ d_{F}^i=d_{G}^i \circ f_i$.
The  sheaf (of $\mathbb{C}$-vector spaces)
of maps in $\Omega_X$-dgmod is 
denoted
\[
\mathcal{H}om_{\Omega_X}({F},{G}).
\]

\noindent
It is proved in \cite[\S 2]{MR0310287} that the category $\Omega_X$-dgmod
is an abelian category 
which satisfies axioms AB5 and AB3* of \cite{MR0102537} and it has enough injectives.

\

\noindent
The following definitions are standard in the context of 
differential graded algebras, \cite{MR1258406}. 
Let $\Omega_X^{\sharp}$-mod be the category of 
coherent graded modules over $\Omega_X$ (without differentials).
Any $\Omega_X^{\sharp}$-Module ${F}$ has a 
grading ${F}=\oplus_{i\in\mathbb{Z}}{F}^i$ such that
\[
\wedge:\Omega_X^{i}\otimes_{\OO_X}{F}^j\to{F}^{i+j}.
\]
Given that ${F}$ is coherent, we have ${F}^{i}=0$ for $|i|>>0$.

\

\noindent
Let us define
$\mathcal{H}omgr_{\Omega_X^{\sharp}}({F},{G})$,
the graded sheaf of $\Omega_X^{\sharp}$-homomorphisms 
between two $\Omega_X^{\sharp}$-Modules ${F}$ and 
${G}$, as
\begin{equation}
\mathcal{H}omgr_{\Omega_X^{\sharp}}({F},{G}):=
\bigoplus_{k\in\mathbb{Z}}
\mathcal{H}om_{\Omega_X^{\sharp}}({F},{G}[k]), \label{homgr}
\end{equation}
where ${G}[k]^j={G}^{j+k}$, $d_{{G}[k]}=(-1)^kd_{{G}}$  and $\mathcal{H}om_{\Omega_X^{\sharp}}({F},{G}[k])$
is the sheaf of 
$\Omega_X^{\sharp}$-linear maps ${F}\to{G}[k]$.  Any object in $\Omega_X$-dgmod is called an $\Omega_X^\sharp$-Module.

\

\noindent
If ${F}$ and ${G}$ are differential graded $\Omega_X$-Modules, one gives
to  $\mathcal{H}omgr_{\Omega_X^{\sharp}}({F},{G})$
a differential $d$ making it a differential graded $\Omega_X$-Module,
\begin{equation}
df:=d_{{G}} \circ f-(-1)^k f \circ d_{{F}},\quad
f\in 
\mathcal{H}om_{\Omega_X^{\sharp}}({F},{G}[k])   \label{differential}
\end{equation}
In this case we may depict the complex $(\mathcal{H}omgr_{\Omega_X^{\sharp}}({F},{G}), d)$ in the usual way:
\begin{equation}
\dots \to \mathcal{H}om_{\Omega_X^{\sharp}}({F},{G}[k-1]) \to  \mathcal{H}om_{\Omega_X^{\sharp}}({F},{G}[k]) \to  \mathcal{H}om_{\Omega_X^{\sharp}}({F},{G}[k+1]) \to \dots, \label{homgr2}
\end{equation}
We denote 
\begin{equation}
B(\mathcal{H}omgr_{\Omega_X^{\sharp}}({F},{G}),d), \
Z(\mathcal{H}omgr_{\Omega_X^{\sharp}}({F},{G}),d), \
H(\mathcal{H}omgr_{\Omega_X^{\sharp}}({F},{G}),d), \label{boundaries}
\end{equation}
the boundaries (image of $d$), the cicles (kernel of $d$), and the homology $H = Z/B$ of the
differential complex $(\mathcal{H}omgr_{\Omega_X^{\sharp}}({F},{G}),d)$. 
These are $\Omega_X^\sharp$-Modules. Notice that 
\begin{equation}
Z^0(\mathcal{H}omgr_{\Omega_X^{\sharp}}({F},{G}),d)=
\mathcal{H}om_{\Omega_X}({F},{G}).  \label{Z^0}
\end{equation}
In particular, we have an exact sequence of sheaves of complex vector spaces
\begin{equation}
\xymatrix{
0 \ar[r] &\mathcal{H}om_{\Omega_X}({F},{G})  \ar[r] &\mathcal{H}om_{\Omega_X^{\sharp}}({F},{G})  \ar[r]^d &\mathcal{H}om_{\Omega_X^{\sharp}}({F},{G[1]})
}
\label{cicles}
\end{equation}
where $d(f) = [d, f]$ as in \ref{differential}.
This exact sequence is functorial in $F$ and in $G$.

\

\noindent
The tensor product in $\Omega_X$-dgmod is defined as
\[
{F}\otimes_{\Omega_X} {G}:=
({F}\otimes_{\Omega_X^{\sharp}} {G},d),
\quad
d(f\otimes g):=df\otimes g+(-1)^{\overline{f}}f\otimes dg,
\]

\begin{remark}
In \cite{MR0310287} the authors define
the Hyperext, denoted $\underline{\underline{\mathcal{E}xt}}$,
as the right derived functors of
$(\mathcal{H}omgr_{\Omega_X^{\sharp}}(-,-),d)$.
Given that $\Omega_X$-dgmod is an abelian category, 
it is also possible to 
define right derived functors of
$\mathcal{H}om_{\Omega_X}(-,-)$.
Notice that these functors are different,
\[
\underline{\underline{\mathcal{E}xt}}^0(\Omega_X,{G})= H^0(G,d),
\quad
\mathcal{E}xt_{\Omega_X}^0(\Omega_X,G)=
Z^0(G,d).
\]
See \cite[3.3]{MR0310287} for the first equality. 
\qed
\end{remark}

\

\begin{proposition}\label{local-global}
Let ${F},{G}$ be differential graded $\Omega_X$-Modules such that 

\noindent
$\mathcal{H}om_{\Omega_X^{\sharp}}({F},{G}[k])=0$ for $k < 0$.
Then, there exists an exact sequence
of $\C$-vector spaces,
\[
0\to 
H^1(X,\mathcal{H}om_{\Omega_X}({F},{G}))\to
\Ext_{\Omega_X}^{1}({F},{G})\to
H^0(X,\mathcal{E}xt_{\Omega_X}^{1}({F},{G}))\to
\]
\[
H^2(X,\mathcal{H}om_{\Omega_X}({F},{G}))\to
\Ext^{2}_{\Omega_X}({F},{G}).
\]
\end{proposition}

\begin{proof}
The result follows by applying Grothendieck spectral sequence
to the composition of $H^0(X,-)$ and 
$\mathcal{H}om_{\Omega_X}({F},-)$, see 
\cite[Thm. 2.4.1]{MR0102537}.
Let ${J}$ be an injective differential graded $\Omega_X$-Module. Let us prove
that $\mathcal{H}om_{\Omega_X}({F},{J})$ is flasque.
We follow \cite[Prop. 4.1.3]{MR0102537}.
Let $j:U\hookrightarrow X$ be the inclusion of some open subset $U$.
Let us denote by ${F}|_U:=j^*{F}$.

\

\noindent
Any local $\OO_X$-map ${F}|_U\to{J}|_U$ is the same as a
global $\OO_X$-map $j_{!}({F}|_U)\to{J}$. 
Given that ${J}$
is injective (also as an $\OO_X$-Module)
and $j_{!}({F}|_U)\subseteq{F}$, this map extends
to a $\OO_X$-linear map ${F}\to{J}$. 
Now, if a given local $\OO_X$-map commutes with the differentials,
then the extension also commutes with the differentials because
$J$ is injective in $\Omega_X$-dgmod.
Hence, 
$\mathcal{H}om_{\Omega_X}({F},{J})$ is flasque.
\end{proof}

\medskip
  
\subsection{Pfaff ideals.} \label{pfaff ideals}

\

\noindent
Let $X$ be a smooth irreducible algebraic variety.
Now we focus on exterior differential ideals, that is, differential graded submodules $\II \subset \Omega_X$.
For $r \in \N$ denote $\II^r \subset \Omega^r_X$ the homogeneous component of degree $r$ of $\II$. See \cite{MR1083148} for more information.

\

\begin{definition} \label{def pfaff ideal}
Let $\II \subset \Omega_X$ be a differential graded ideal.

a) We say that $\II$ is  \emph{generated in degree one}, or that it is a \emph{Pfaff ideal},
if it is generated, as a differential graded ideal, by its homogeneous component of degree one $\II^1 \subset \Omega_X^1$. 
Equivalently, $\II^r = \Omega_X^{r-1} \wedge \II^1 + \Omega_X^{r-2} \wedge d(\II^1)$ for all $r$. 
Notice that, conversely, if $\II^1$ is a coherent subsheaf of $\Omega_X^1$, the last formula defines a differential graded ideal $\II$ generated by $\II^1$.

b) We say that $\II$ is an \emph{integrable Pfaff ideal} if it is generated in degree one and $d(\II^1) \subset \Omega_X^{1} \wedge \II^1$. 
Notice that in this case we have $\II^r = \Omega_X^{r-1} \wedge\II^1$ for all $r$. 

c) If $\II \subset \Omega_X$ is a differential graded ideal, we say it is a \emph{singular integrable Pfaff ideal} 
if $\II$ is generated in degree one and $d(\II^1|_U) \subset \Omega_U^{1} \wedge \II^1|_U$
for some (Zariski) open dense subset $U \subset X$. Notice that then $\II|_U$ is an integrable  Pfaff ideal on $U$.

d) Similarly, a differential graded $\Omega_X$-Module $F$ is \emph{generated in degree one}
if $F^r = \Omega_X^{r-1} \wedge F^1 + \Omega_X^{r-2} \wedge d(F^1)$ for all $r$. 
The other definitions above extend similarly to  $\Omega_X$-Modules.

e) If $\II \subset \Omega_X$ is a differential graded ideal, we say that $\II$ is locally-free (resp. \emph{reflexive}) if it is locally-free as $\OO_X$-Module (resp. reflexive, that is, isomorphic to its double dual). Equivalently, if $\II^r$ is a locally-free $\OO_X$-Module (resp. reflexive) for all $r$.
\end{definition}

\

\begin{definition} \label{singular scheme}
\
Let $F$ be a coherent sheaf of $\OO_X$-Modules. There exists a non-empty Zariski open $U \subset X$ such that $F|_U$ is locally free, of rank $r(F)$.
Taking the largest such $U$, one calls $S = X - U$ the \emph{singular set} $\text{Sing}(F)$ of $F$. The closed set $\text{Sing}(F)$ may be characterized as
the union for $j  > 0$ of the supports of the coherent sheaves $\EE xt_{ \OO_X}^j(F, \OO_X)$, see \cite{okonek1980vector} or \cite{MR0463157} (III, Ex. 6.6).
Alternatively, $\text{Sing}(F)$ is the closed set defined by the $r(F)$-th Fitting ideal of $F$, see \cite{MR1322960}, \cite{lang2004algebra},  or \cite{bourbaki2007elements}, \S 3, Exercise 8.

\noindent
If $\II \subset \Omega_X$ is a differential ideal, from the exact sequence $0 \to \II \to \Omega_X \to  \Omega_X/\II \to 0$ of $\OO_X$-Modules, and assuming $X$ smooth, one easily obtains 
$\text{Sing}(\II) \subset \text{Sing}(\Omega_X/\II)$, which is usually a strict inclusion. We shall denote 
$$S(\II) = \text{Sing}(\Omega_X/\II),$$ 
and call this set the singular set of the ideal $\II$, if there is no danger of confusion.
\end{definition}

\

\noindent
For instance, if $\omega_1, \dots, \omega_q$ are one-forms in $X$ we denote 
\begin{equation}
\II = \langle\omega_1, \dots, \omega_q\rangle
\label{ideal generated}
\end{equation}
the differential graded ideal generated by $\omega_1, \dots, \omega_q$.
Then $\II$ is an integrable Pfaff ideal when $d\omega_i \in \sum_j \Omega_X^1 \wedge \omega_j$ for all $i$, that is, locally,
\begin{equation}
d\omega_i  = \sum_j \mu_{ij} \wedge \omega_j  
\label{frobenius}
\end{equation}
for some one-forms $\mu_{ij}$. This implies the Frobenius condition $\omega  \wedge d\omega_i = 0$ for all $i$, where 
$\omega = \omega_1 \wedge \dots \wedge \omega_q$. Conversely, on the open set $U$ where $\omega_1, \dots, \omega_q$ are linearly independent over $\OO_X$ the Frobenius condition implies
that $\II|_U$ is a integrable Pfaff ideal on $U$. If $\omega_1, \dots, \omega_q$ are linearly independent over $\OO_X$ on a dense open set then the singular set of $\II$ is the set of zeros of the $q$-form $\omega = \omega_1 \wedge \dots \wedge \omega_q$.

\

\begin{definition} \label{analytic sheaf} As in \cite{serre1956geometrie}, and \cite{grothendieck1956faisceaux}, denote $X^h$ the complex analytic variety associated to $X$, and $\OO^h_X$ its sheaf of holomorphic functions.
If $F$ is an $\OO_X$-Module, $F^h$ denotes the associated sheaf of $\OO^h_X$-Modules on $X^h$. As shown in the references above, $F \mapsto F^h$ is an exact functor.
If $F$ is a graded $\Omega_X$-Module, it easily follows that 
$F^h$ is a graded $\Omega_X^h$-Module. If ${F_1}$ and ${F_2}$ are coherent $\OO_X$-Modules and $D: F_1 \to F_2$ is a differential operator of order $r$ then there exists
a canonically defined differential operator of order $r$, $D^h: F_1^h \to F_2^h$, extending $D$; see \cite{deligne1969equation}, Chapter II, (6.6). 
Thus, if $F$ is a differential graded $\Omega_X$-Module,  we may naturally define  a differential graded $\Omega_X^h$-Module $F^h$. We obtain an exact functor $F \mapsto F^h$ from differential graded $\Omega_X$-Modules to differential graded $\Omega_X^h$-Modules.
If ${F}$ and ${G}$ are differential graded $\Omega_X$-Modules, we denote
$\mathcal{H}om_{\Omega_X}({F},{G})^h$ the sheaf of homomorphisms $F^h \to G^h$ of differential graded $\Omega^h_X$-Modules that commute with the differentials. As usual, we shall denote by the same letter $d$ the exterior derivative of differential forms, in the algebraic De Rham complex $(\Omega_X, d)$ and in the analytic De Rham complex $(\Omega^h_X, d)$.
\end{definition}

\newpage

\section{Deformations of exterior differential ideals.}

\

\subsection{Infinitesimal deformations.}

\
 
\begin{definition} \label{def family}
A \emph{family of differential graded ideals parametrized by a scheme $S$} consists of a morphism of schemes  $f: \XX \to S$ and  
a coherent subsheaf  $\II \subset \Omega_{\XX/S}$ which is a differential graded ideal of the differential graded algebra 
$\Omega_{\XX/S}$, and such that $\Omega_{\XX/S}/\II$ is flat over $S$.

\

\noindent
In this article we shall focus on families of differential graded ideals as above such that $\XX = X \times S$, and $f$ is the canonical projection,
that is, the ambient variety $X$ does not deform.
We refer to these as families of differential graded ideals \emph{on $X$}.
\end{definition}

\

\noindent
In order to work out the infinitesimal deformation theory of exterior differential ideals, let us start with two lemmas.

\begin{lemma}\label{trivial-def}
Let $X$ be a algebraic variety over $\C$ and
let $X_A$ be the trivial deformation of $X$ over a local Artin ring $A$.
Then $\Omega_{X}\otimes_{\C} A= \Omega_{X_A|A}$
in $\Omega_{X_A|A}$-dgmod.
\end{lemma}

\begin{proof}
By hypothesis, $\OO_{X_A}=\OO_{X}\otimes_{\C}A$,
Then,
\[
\Omega^1_{X_A}= \Omega^1_{X}\otimes_{\C} A\oplus
\Omega^1_{A}\otimes_{\C}\OO_X
\Longrightarrow
\Omega^1_{X_A|A}= \Omega^1_{X}\otimes_{\C} A.
\]
It follows that $\Omega_{X_A|A}=\Omega_{X}\otimes_{\C} A$
where the differential maps the elements of $A$ to zero.
\end{proof}

\

\begin{lemma}
Let $X$ be an algebraic variety over $\C$ and
let $X_A$ be a deformation of $X$ over a local Artin ring $A$.
Then $\Omega_{X_A|A}\otimes_{A} \C= \Omega_{X}$
in $\Omega_{X}$-dgmod.
\end{lemma}

\begin{proof}
By hypothesis, 
$\OO_{X}=\OO_{X_A}\otimes_{A}\C$.
Then,
\[
\Omega^1_{X}= \Omega^1_{X_A|A}\otimes_{A}\C\oplus
\Omega^1_{\C|A}\otimes_A\OO_{X_A}=\Omega^1_{X_A|A}\otimes_{A}\C.
\]
It follows that $\Omega_{X}=\Omega_{X_A|A}\otimes_{A} \C$
in $\Omega_X$-dgmod.
\end{proof}

\medskip

\noindent
Consider the trivial extension,
\[
\xymatrix{
0\ar[r]&\C\ar[r]^{\varepsilon}&\C[\varepsilon]\ar[r]&\C\ar[r]&0.
}
\]
A first order deformation of $\II$ is a differential ideal $\II_\varepsilon$ over a
deformation $X_\varepsilon$ of $X$ such that 
$\Omega_{X_\varepsilon|\C[\varepsilon]}/\II_\varepsilon$
is flat over $\C[\varepsilon]=\C[t]/(t^2)$ and such that
$\C\otimes_{\C[\varepsilon]}\Omega_{X_\varepsilon|\C[\varepsilon]}/\II_\varepsilon
=\Omega_{X}/\II$.
For simplicity, let us denote $\Omega:=\Omega_X$ and $\Omega_{\varepsilon}:=
\Omega_{X_\varepsilon|\C[\varepsilon]}$. 

\begin{proposition} \label{first order deformations}
Let $\II$ be a differential ideal over an algebraic variety $X$.
Then, the vector space 
\[
\Hom_{\Omega}(\II,\Omega/\II)
\]
parametrizes first order deformations of $\II$ on $X$.
\end{proposition}

\begin{proof}
We follow \cite{MR2583634}, Proposition (2.3). Let $\II_{\varepsilon}$ be a first order deformation of $\II$.
Let us tensor (over $\C[\varepsilon]$) the trivial extension with
$\II_{\varepsilon}$, 
$\Omega_{\varepsilon}$ and $\Omega_{\varepsilon}/\II_{\varepsilon}$
to obtain the following commutative diagram of exact rows and
columns in $\Omega_\varepsilon$-dgmod,
\[
\xymatrix{
&0\ar[d]&0\ar[d]&0\ar[d]&\\
0\ar[r]&\II\ar[r]\ar[d]&\II_{\varepsilon}\ar[r]\ar[d]&\II\ar[r]\ar[d]&0\\
0\ar[r]&\Omega\ar[r]\ar[d]&\Omega_{\varepsilon}\ar[r]\ar[d]&\Omega\ar[r]\ar[d]&0\\
0\ar[r]&\Omega/\II\ar[r]\ar[d]&\Omega_{\varepsilon}/\II_{\varepsilon}\ar[r]\ar[d]&\Omega/\II\ar[r]\ar[d]&0\\
&0&0&0&\\
0\ar[r]&\C\ar[r]&\C[\varepsilon]\ar[r]&\C\ar[r]&0
}
\]

\noindent
Take $\omega\in\II^k$ and lift it to an element in $\II_{\varepsilon}^k$, 
$\omega+\varepsilon\eta$, where $\eta\in \Omega^k$.
Two liftings differ by an element in $\II^k$, hence the class of $\eta$ is a well-defined element in $\Omega^k/\II^k$.
Then, we have a degree zero $\Omega_X^{\sharp}$-map $\varphi:\II\to \Omega/\II$, 
$\omega\mapsto \overline{\eta}$. Given that the maps in the diagram commute with
the differentials (recall that $d\varepsilon=0$)
if $\omega+\varepsilon\eta$ lifts $\omega$, then
$d\omega+\varepsilon d\eta$ lifts $d\omega$. It follows that
$\varphi\in \Hom_{\Omega}(\II,\Omega/\II)$.

\

\noindent
Conversely, given $\varphi\in\Hom_{\Omega}(\II,\Omega/\II)$ define
$\II\times_\varphi \Omega$ using the following diagram in $\Omega$-dgmod,
\[
\xymatrix{
0\ar[r]&\II\ar[r]^-{\cdot\varepsilon}\ar@{=}[d]&\II\times_\varphi \Omega \ar[r]^-{\pi_1}\ar[d]&\II\ar[r]\ar[d]^\varphi&0\\
0\ar[r]&\II\ar[r]&\Omega\ar[r]&\Omega/\II\ar[r]&0
}
\]
Specifically,
\begin{equation}
\II\times_\varphi \Omega:=\{\omega+\varepsilon\eta\,\colon\,\varphi(\omega)=\overline{\eta}\}\subseteq \Omega_{\varepsilon}. \label{deformation formula}
\end{equation}
It is clearly a differential ideal in $\Omega_{\varepsilon}$.
Replacing $\II_{\varepsilon}$ by $\II\times_\varphi \Omega$ in the first diagram 
above, we obtain from the definition of $\II\times_\varphi \Omega$
that the first row is exact 
and by the Snake Lemma, the third row is also exact.
Hence, $\Omega_{\varepsilon}/\II_{\varepsilon}$ is flat over $\C[\varepsilon]$ (see \cite{MR2583634}, Proposition 2.2) and
$\C\otimes_{\C[\varepsilon]} \Omega_{\varepsilon}/\II_{\varepsilon}=\Omega/\II$.
\end{proof}

\medskip

\noindent
Let
\[
e:\quad 0\to J\to B\to A\to 0,\quad e\in \Ex(A,J)
\]
be a ring extension with $J^2=0$. In particular, $J$ has a structure
of $A$-module.
The zero element in $\Ex(A,J)$ corresponds to the trivial extension $B=A[J]$, \cite{MR0217083} \S 18, \cite{MR2583634}.

\

\noindent
As before, denote $X_B$, resp.  $X_A$, the trivial deformation of $X$ over $B$, resp. A. Let $\II_A$ be a differential ideal in $\Omega_A:=\Omega_{X_A|A}$.
From Definition  \ref{def family}, a deformation  of $\II_A$ over $B$ is a differential ideal $\II_B$ in 
$\Omega_B:=\Omega_{X_B|B}$ such that $\Omega_B/\II_B$
is flat over $B$ and $(\Omega_B/\II_B)\otimes_B A=\Omega_A/\II_A$.

 \

\begin{proposition}\label{def-inf}
Let $\II_A$ be a differential ideal over $X_A$.
Then, the set of deformations of $\II_A$ over $B$
is a pseudotorsor for the $A$-module
\[
\Hom_{\Omega_A}(\II_A,J\otimes_A \Omega_A/\II_A).
\]
In other words, if there exists a deformation $\II_B$,
the natural action 
in the set of deformations is free and transitive.
\end{proposition}

\begin{proof}
We follow the proof of \cite[Theorem 6.2]{MR2583634}.
Suppose there exists a deformation $\II_B$ and 
consider the following diagram of exact rows and columns in $\Omega_B$-dgmod,
\[
\xymatrix{
&0\ar[d]&0\ar[d]&0\ar[d]&\\
0\ar[r]&J\otimes_A \II_A\ar[r]\ar[d]&\II_B\ar[r]\ar[d]&\II_A\ar[r]\ar[d]&0\\
0\ar[r]&J\otimes_A \Omega_A\ar[r]\ar[d]&\Omega_B\ar[r]\ar[d]&\Omega_A\ar[r]\ar[d]&0\\
0\ar[r]&J\otimes_A \Omega_A/\II_A\ar[r]\ar[d]&\Omega_B/\II_B\ar[r]\ar[d]&\Omega_A/\II_A\ar[r]\ar[d]&0\\
&0&0&0&\\
0\ar[r]&J\ar[r]&B\ar[r]&A\ar[r]&0
}
\]

\noindent
Let $\II_B',\II_B''\subseteq\Omega_B$ be two 
deformations of $\II_A$ over $B$ (possibly different from $\II_B$). 
Let $\omega\in \II_A^k$ and consider two
liftings $\omega'\in \II_B'^k$ and $\omega''\in \II_B''^k$. 
The difference, $\omega'-\omega''$ is an element in $\Omega_B^k$
and maps to zero in $\Omega_A^k$. Hence, $\omega'-\omega''\in J\otimes_A \Omega_A^k$.
Let us call 
$\varphi(\omega)=\overline{\omega'-\omega''}\in J\otimes_A \Omega_A^k/\II_A^k$.
Given that $\omega'$ and $\omega''$ are defined up to an element in $J\otimes_A \II_A^k$,
$\varphi(\omega)$ is a well defined degree zero $\Omega_{X_A}^{\sharp}$-homomorphism
and commutes with the differentials. Then, 
$\varphi\in\Hom_{\Omega_A}(\II_A,J\otimes_A \Omega_A/\II_A)$.
We say that $\varphi$ is defined by $\II_B'$ and $\II_B''$.

\

\noindent
Assume now that we have a deformation $\II_B'$ and a 
morphism $\varphi:\II_A\to J\otimes_A \Omega_A/\II_A$ of $\Omega_A$-Modules.
Let us define the action of $\varphi$ on $\II_B'$ by constructing
another deformation $\II_B''$ in such a way
that the map associated to $\II_B'$ and $\II_B''$ is precisely $\varphi$.
Let us call 
$p:\Omega_B\to \Omega_A$ the pull-back map
induced by $p:B\to A$. Then,
\[
\II_B'':=\left\{ \omega\in p^{-1}(\II_A)\,\colon\,
\overline{\omega-\eta}=\varphi(p(\omega))
\quad\forall\eta\in p^{-1}(p(\omega))\cap \II_B'  \right\}.
\]
A simple computation shows that $\II_B''$ is a deformation of $\II_A$ over $B$
and by construction, $\varphi$ is defined by $\II_B'$ and $\II_B''$.

\

\noindent
Note finally that if $\II_B'$, $\II_B''$ and $\II_B'''$ are three 
different deformations of $\II_A$, and
if $\varphi_1$ is defined by $\II_B',\II_B''$ as above, 
$\varphi_2$ defined by $\II_B'',\II_B'''$,
and $\varphi_3$ defined by $\II_B',\II_B'''$, 
then $\varphi_3=\varphi_1+\varphi_2$. Thus
the operation $(\II_B,\varphi)\to \II_B'$ is a free and transitive action of
the group 
$\Hom_{\Omega_A}(\II_A,J\otimes_A \Omega_A/\II_A)$
on the set (possibly empty) of deformations of $\II_A$ over $B$.
\end{proof}

\medskip

\subsection{Obstruction theory.}

\

\noindent
Let $X$ be an algebraic variety over $\C$ 
and let $\II\subseteq\Omega_X$ be a differential. Assume that $\II_A$ is a deformation of $\II$
over $A$ and let $e\in \Ex(A,J)$ be a commutative ring extension, 
\[
e:\quad 0\to J\to B\to A\to 0,\quad J^2=0.
\]
Let us denote $\Omega_A:=\Omega_{X_A|A}$
and $\Ext_{\Omega_A}(U,V)$  the 
$\C$-vector space of extensions between $V$ and $U$
in $\Omega_A$-dgmod.

\begin{proposition}\label{obst}
With the notation above, there exists an element 
\[
ob(e)\in\Ext_{\Omega_A}(\II_A,J\otimes_A \Omega_A/\II_A)
\]
such that $ob(e)=0$ if and only if there exists a deformation of $\II_A$ over $B$.
\end{proposition}

\begin{proof}
We follow \cite[Th. 6.4.5]{MR2222646}. Let $e\in \Ex(A,J)$ be an
extension as above. Assume that we have a deformation $X_B$ of $X_A$ over $B$.
Recall from Lemma \ref{trivial-def}
that if $X_B$ is the trivial deformation, then 
$\Omega_B= B\otimes_A\Omega_A$.
Let us construct $ob(e)$ by considering the following diagram in
$\Omega_B$-dgmod,
\[
\xymatrix{
&0\ar[d]& &0\ar[d]&\\
 &J\otimes_A \II_A\ar[d]& &\II_A \ar[d]& \\
0\ar[r]&J\otimes_A \Omega_A\ar[r]\ar[d]&\Omega_B\ar[r]&\Omega_A\ar[r]\ar[d]&0\\
 &J\otimes_A \Omega_A/\II_A\ar[d]& &\Omega_A/\II_A \ar[d]&\\
&0& &0&\\
0\ar[r]&J\ar[r]&B\ar[r]&A\ar[r]&0
}
\]
Since $\Omega_A/\II_A$ is flat over $A$, 
$\Tor_1^{A}(J,\Omega_A/\II_A)=0$ and it follows that the first column is exact.
The exactness of the row follows by taking $e\otimes_A\Omega_B$. Then, the previous diagram has exact rows and columns.

\

\noindent
Consider now the induced $\Omega_B$-homomorphisms
$\alpha:J\otimes_A \II_A\to \Omega_B$ and 
$\beta:\Omega_B\to \Omega_A/\II_A$. 
Clearly, $\im(\alpha)\subseteq\ker(\beta)$
and we can define ${M}:=\ker(\beta)/\im(\alpha)$.
It is easy to see that the following sequence is exact,
\[
\xymatrix{
ob(e):&0\ar[r]&J\otimes_A \Omega_A/\II_A\ar[r]&{M}\ar[r]&\II_A\ar[r]&0.
}
\]
Given that $J.{M}=0$, the object ${M}$ in $\Omega_B$-dgmod 
is also in $\Omega_A$-dgmod, hence
\[
ob(e)\in\Ext_{\Omega_A}(\II_A,J\otimes_A \Omega_A/\II_A).
\]

\noindent
Let us see now that deformations of $\II_A$ over $B$ are in one-to-one correspondence 
with splittings of $ob(e)$.
Given a splitting $\psi:\II_A\to{M}$, 
consider $X_B$ as the trivial deformation, hence $\Omega_B= B\otimes_A\Omega_A$
and let $\II_B$ be the inverse image of $\psi(\II_A)$ in $\ker(\beta)\subseteq 
\Omega_B$. Then, $\II_B$ is a differential ideal of $\Omega_B= B\otimes_A\Omega_A$
and a deformation of $\II_A$ over $B$.
Conversely, given a deformation $\II_B$, 
the submodule $\II_B/\im(\alpha)\subseteq 
{M}$ maps isomorphically into $\II_A$, therefore
defining a splitting of ${M}\to \II_A$.
\end{proof}

\medskip

\

\begin{proposition} \label{commute ext}
If $J\cdot \mathfrak{m}_A =0$, then we can replace $\Ext_{\Omega_A}(\II_A,J\otimes_A \Omega_A/\II_A)$ with $\Ext_{\Omega}(\II,J\otimes_A \Omega_A/\II_A)$ in the statement of Proposition \ref{obst}. And also,
$$\Ext_{\Omega}(\II, J \otimes_A \Omega_A/\II_A) = \Ext_{\Omega}(\II,\Omega/\II) \otimes_{\mathbb{C}} J = \Ext^1_{\Omega}(\II,\Omega/\II) \otimes_{\mathbb{C}} J.$$
\end{proposition}
\begin{proof}
We follow \cite[Prop.6.4.7]{MR2222646}.  Given that $\II_A$ is flat over $A$, it follows $\mathrm{Tor}_1^A(\II_A,\C)=0$ and then, we can apply $-\otimes_A\C$ to an extension 
in $\Ext_{\Omega_A}(\II_A,J\otimes_A \Omega_A/\II_A)$
to get the following diagram with exact rows and exact column
\[
\xymatrix{
&&&0\ar[d]\\
&&&\II_A\otimes_A \mathfrak{m}_A\ar[d]\\
0\ar[r] &J\otimes_A \Omega_A/\II_A\ar@{=}[d]\ar[r]&M\ar[d]\ar[r]& \II_A \ar[d]\ar[r]&0\\
0 \ar[r]&J\otimes_A \Omega_A/\II_A\ar[r]&M\otimes_A\C\ar[r]& \II \ar[d]\ar[r]&0\\
&&&0\\
}
\]
The hypothesis $J\cdot \mathfrak{m}_A =0$ implies  $J\otimes_A \C=J$ and then
in the bottom row we have, $(J\otimes_A \Omega_A/\II_A)\otimes_A\C = J\otimes_A \Omega_A/\II_A$. The exactness of the column follows from the fact that $I_A$ is flat over $A$.
Now, by the Snake Lemma, we can complete the diagram in the following way
\[
\xymatrix{
&&0\ar[d]&0\ar[d]\\
&&\II_A\otimes_A \mathfrak{m}_A\ar[d]\ar@{=}[r]&\II_A\otimes_A \mathfrak{m}_A\ar[d]\\
0\ar[r] &J\otimes_A \Omega_A/\II_A\ar@{=}[d]\ar[r]&M\ar[d]\ar[r]& \II_A \ar[d]\ar[r]&0\\
0 \ar[r]&J\otimes_A \Omega_A/\II_A\ar[r]&M\otimes_A\C\ar[r]\ar[d]& \II \ar[d]\ar[r]&0\\
&&0&0\\
}
\]
This shows that the middle row splits if and only if the bottom row splits.

\

\noindent
To prove the second statement, let $e\in \Ex(A,J)$ be an extension, where $A$ is a local Artinian $\mathbb{C}$-algebra.
Then, the ideal $J$ has a structure of a $\mathbb{C}$-vector space
and from \cite[Prop.C3.1(ii)]{MR1863391} we can take $J$ out,
\[
\text{Hom}_{\Omega}(\II,J\otimes_{\C}\Omega/\II)=
\text{Hom}_{\Omega}(\II,\Omega/\II)\otimes_{\mathbb{C}} J,
\]
\[
\Ext^1_{\Omega}(\II,J\otimes_{\C}\Omega/\II)=
\Ext^1_{\Omega}(\II,\Omega/\II)\otimes_{\mathbb{C}} J.
\]
The isomorphism between $\Ext_{\Omega}^1$ and the Yoneda group $\Ext_{\Omega}$ is standard, see for example \cite[Ch.III, Th.2.4]{MR0346025} or \cite[Def.2.4]{MR3640821}.

\end{proof}

\medskip

\begin{corollary} \label{tangent obstruction}
Let $X$ be an algebraic variety over $\C$ 
and let $\II\subseteq\Omega_X$ be an exterior differential
ideal.
The functor of deformations of $\II$ on $X$ has a tangent-obstruction theory with
\[
T_ 1 = \Hom_{\Omega}(\II,\Omega/\II)
,\quad
T_2 = \Ext^1_{\Omega}(\II,\Omega/\II).
\]
\end{corollary}

\begin{proof}
We refer to \cite{MR2222646}, Chapter 6, for tangent-obstruction theories. 
Our statement then follows from Propositions \ref{first order deformations} to \ref{commute ext}.
\end{proof}

\medskip

\

\begin{remark} \label{smaller obstruction}
Notice in the proof of Proposition \ref{obst} that
\[
\Omega_B\cong\Omega_A\oplus J\otimes_A\Omega_A
\]
as $A$-modules. Then,
\[
\ker(\beta)=\{\omega+j\eta\,\colon\,\overline{\omega}=0\}=
\II_A\oplus J\otimes_A\Omega_A,\quad
\im(\alpha)=0\oplus J\otimes_A\II_A
\]
and as $A$-modules we have,
\[
M=\ker(\beta)/\im(\alpha)\cong \II_A\oplus J\otimes_A\Omega_A/\II_A.
\]
Then, $ob(e)$ splits in $\Ext_{A}(\II_A,J\otimes_A \Omega_A/\II_A)$
and then,
\[
ob(e)\in H^1(\mathcal{H}omgr_{\Omega_X^\sharp}(\II_A,J\otimes_A \Omega_A/\II_A),d).
\]
Hence in Corollary  \ref{tangent obstruction} we may replace $\Ext^1_{\Omega}(\II,\Omega/\II)$ 
by the subspace 
\[
H^1(\mathcal{H}omgr_{\Omega_X^\sharp}(\II, \Omega_X/\II),d)
\]
as an alternative smaller obstruction theory; see Proposition \ref{torsion-map}.
\end{remark}

\

\begin{remark}\label{artin}
Proposition \ref{def-inf} and Proposition \ref{obst} are particular
cases of the general theory of abstract Hilbert 
schemes, \cite[Prop.E2.4]{MR1863391}. 
The authors in \cite{MR1863391} define a deformation functor $F$ and compute its 
tangent-obstruction theory. 
Using the notation
$R_0=\mathbb{C}$, $R=A$, $R'=B$, 
${C}_{R_0}=\Omega_X$-dgmod, 
${C}_{R}=\Omega_{X_A|A}$-dgmod,
${C}_{R'}=\Omega_{X_B|B}$-dgmod,
$U_0=\II$, $U=\II_A$, 
$V_0=\Omega/\II$ and $W_0=J\otimes_{\mathbb{C}}\Omega/\II$,
the authors showed that the space of deformations of $\II_A$ over $B$ is 
\[
\Hom_{C_{R_0}}(U_0,W_0) = 
\Hom_{\Omega}(\II,J\otimes_{\mathbb{C}}\Omega/\II)
\]
and that the obstruction space is
\[
\Ext^1_{C_{R'}}(U,W_0).
\] 
If $J\cdot U=0$, it follows 
from 
\cite[Prop.C7.1]{MR1863391} that $U$ is also $R'$-flat and this implies from 
\cite[Prop.C3.1(v)]{MR1863391} the equality 
\[
\Ext^1_{C_{R'}}(U,W_0) = \Ext^1_{C_{R_0}}(U_0,W_0).
\]
Indeed, by hypothesis the sequence obtained by applying $-\otimes_{R'} U$ to $0\to J\to R'\to R\to 0$ is exact and then, by \cite[Prop.C7.1]{MR1863391}, $U$ is $R'$-flat. Now, apply \cite[Prop.C3.1(v)]{MR1863391} to get
\[
\Ext^1_{C_{R_0}}(U\otimes_{R'}R_0,W_0)=\Ext^1_{C_{R'}}(U,W_0).
\]
Finally, 
\[
\Ext^1_{C_{R_0}}(U_0,W_0)=\Ext^1_{\Omega}(\II,J\otimes_{\C}\Omega/\II).
\]

\qed
\end{remark}

\

\begin{definition} \label{def hilb}
Let us denote by $\Hilb(\Omega_X)$ the 
moduli space of differential ideals $\II\subseteq\Omega_X$
on $X$. Points of $\Hilb(\Omega_X)$ are equivalence classes of
pairs $({F},q)$ such that 
$q:\Omega_X\to{F}$ is a surjective $\Omega_X$-linear
map of $\Omega_X$-Modules.
The equivalence class $({F},q)\sim ({F}',q')$ 
is given by an isomorphism $p:{F}\to{F}'$ such that
$pq=q'$ (equivalently, if $\ker(q)=\ker(q')$). The scheme 
$\Hilb(\Omega_X)$ is locally noetherian and represents
a subfunctor of $\Quot_{\Omega_X}$.

\

\noindent
We use the notation $\II \in \Hilb(\Omega_X)$ to denote
the equivalence class of $(\Omega_X/\II,\pi)$ where $\pi$ is
the quotient map $\pi:\Omega_X\to\Omega_X/\II$.
\end{definition}

\newpage

\section{Twisted exterior differential ideals.}

\

\noindent
In this section we develop an alternative way of dealing with twisted differential forms on a projective variety, based on graded modules and homogeneous coordinates.

\

\noindent
A twisted one-form is a section $w \in H^0(X, \Omega^1_X \otimes L)$ where $L$ is a line bundle on $X$. It corresponds to a homomorphism of $\OO_X$-Modules $\omega: L^{-1} \to \Omega^1_X$. 
Denote $\II^1\subset \Omega^1_X$ the image of $\omega$ and $\II = \langle\omega\rangle$ the ideal of $\Omega_X$ generated by $\II^1$ as in \ref{pfaff ideals} 
Notice that $\II^1 \cong L^{-1}$ as $\OO_X$-Modules. 

\

\noindent
Let $\omega$ be a polynomial differential $k$-form in $\C^{n+1}$, homogeneous of degree $e$. Assume $\omega$ descends to $\PP^n$, that is, the contraction with the radial vector field $i_R(\omega)$ vanishes. Hence $\omega$ is a global section of $ \Omega^1_{\PP^n}(e)$. The exterior derivative $d \omega$ is a differential $k+1$-form in $\C^{n+1}$, homogeneous of degree $e$. But $d \omega$ does not descend to $\PP^n$. In fact, $i_R(d \omega) =  d i_R(\omega) + i_R(d \omega) = L_R(\omega) = e \ \omega \ne 0$.

\

\noindent
In this section we propose a framework to work with twisted differential forms over a projective algebraic variety, via the second multiplication of twisted differential forms  \cite{MR2394437}. 

\

\noindent
Let $X$ be a smooth irreducible projective algebraic variety over the complex numbers, with a very ample invertible sheaf ${L}$.
We assume specifically that $X\subseteq\mathbb{P}^n$ and
${L}=\OO_X(1)$.
The \emph{sheaf of algebras of twisted differential forms} is
the bigraded sheaf of algebras $\BB=\BB_{X}$ given by
\[
\BB_X=\bigoplus_{k,e} \Omega_X^k(e).
\]
where $k, e \in \Z$, $0 \le k \le \dim X$.
The \emph{algebra of twisted differential
forms}, $B=B_X$, is
\[
B_X=\bigoplus_{k,e} H^0(X,\Omega_X^k(e)).
\]
In $\BB_X$ (and on $B_X$) there are two different algebra structures.
The first one is given by the wedge product and
the second is called the \emph{second multiplication},
\[
*:\Omega_X^k(e)\times \Omega_X^{k'}(e')\to
\Omega_X^{k+k'+1}(e+e').
\]

\

\noindent
Given $\omega \in \Omega_X^k(e)$ and 
$\eta \in \Omega_X^{k'}(e')$, $\omega * \eta$ is defined, in $\C^{n+1}$, as
\[
\begin{cases}
\frac{e}{e+e'}\omega\wedge d\eta-(-1)^k \frac{e'}{e+e'}d\omega\wedge \eta&
\text{if }e,e'\neq 0.\\
\hfill 0\hfill 
&\text{if }e=0\text{ or }e'=0.
\end{cases}
\]

\

\noindent
It is proved in \cite[Ch. 2, \S 4B, Prop. 4.4]{MR2394437}
that the second multiplication is well defined ($\omega * \eta$ descends to $\PP^n$), and it is associative 
and graded commutative, 
\[
\omega *\eta=(-1)^{(k+1)(k'+1)}\eta*\omega.
\]

\

\noindent
The following relation (similar to Leibnitz rule) holds for 
$\omega\in\Omega_X^{k}(e)$,
$\eta\in\Omega_X^{k'}(e')$ and 
$\mu\in\Omega_X^{k''}(e'')$,
\[
\omega*(\eta\wedge\mu)= 
\frac{e+e'}{e+e'+e''}(\omega*\eta)\wedge \mu
+(-1)^{k'k''}\frac{e+e''}{e+e'+e''}(\omega*\mu)\wedge\eta.
\]

\

\noindent
A $\BB_X$-Module is a bigraded $\OO_X$-Module
${F}=\oplus_{k,e}{F}^{k}(e)$ with two
module structures with respect to $\wedge$ and to $*$,
\[
\wedge:\Omega_X^{k}(e)\otimes_{\OO_X}{F}^{k'}(e')
\to{F}^{k+k'}(e+e'),\quad
*:\Omega_X^{k}(e)\otimes_{\OO_X}{F}^{k'}(e')
\to{F}^{k+k'+1}(e+e')
\]
satisfying the compatibility condition for 
$\omega\in\Omega_X^{k}(e)$,
$\eta\in\Omega_X^{k'}(e')$ and 
$f\in{F}^{k''}(e'')$,
\[
\omega*(\eta\wedge f)=
\frac{e+e'}{e+e'+e''}(\omega*\eta)\wedge f
+(-1)^{k'(k+1)}\frac{e+e''}{e+e'+e''}\eta\wedge(\omega*f).
\]

\

\noindent
We assume that ${F}^{k}=0$ for $k <0$ and $k>>0$.
A $\BB_X$-homomorphism $f:{F}\to{G}$ between two $\BB_X$-Modules
is $\OO_X$-linear and respects the bigrading, $\wedge$ and $*$.
The sheaf of $\BB_X$-homomorphisms between ${F}$ and ${G}$
is a sheaf of $\C$-vector spaces and is denoted 
\begin{equation}
\mathcal{H}om_{\BB_X}({F},{G}) \label{hom sobre omega}.
\end{equation}

\

\noindent
The category of $\BB^\sharp_X$-Modules is the category of 
$\wedge$-coherent bigraded sheaves ${F}
=\bigoplus_{k,e}{F}^{k}(e)$, where
${F}^{k}(e)$ is a $\OO_X$-Module and
\[
\wedge:\Omega_X^{k}(e)\otimes_{\OO_X}{F}^{k'}(e')
\to{F}^{k+k'}(e+e').
\]
The set of $\BB^\sharp_X$-linear maps
is bigraded and a $\BB^\sharp_X$-Module,
\[
\mathcal{H}omgr_{\BB^{\sharp}_X}({F},{G})=
\bigoplus_{k,e}
\mathcal{H}om_{\BB^{\sharp}_X}({F},{G}[k, e]),
\]
where $\mathcal{H}om_{\BB^{\sharp}_X}({F},{G}[k, e])$
is the $\OO_X$-Module of maps of bidegree $(k,e)$, 
\[
\varphi(\omega\wedge f)=(-1)^{kk'}\omega\wedge\varphi(f),\quad
\varphi\in\mathcal{H}om_{\BB^{\sharp}_X}({F},{G}[k, e]),
\omega\in\Omega_X^{k'}(e'),f\in{F}.
\]
Notice that any $\BB_X$-Module is a $\BB^\sharp_X$-Module and
$\mathcal{H}om_{\BB_X}({F},{G})
\subseteq
\mathcal{H}om_{\BB^{\sharp}_X}({F},{G})$.

\

\noindent
If ${F}$ and ${G}$ are $\BB_X$-Modules it is possible
to give a $*$-structure to $\mathcal{H}omgr_{\BB^{\sharp}_X}({F},{G})$,
\[
(\omega*\varphi)(f):=\omega*_{{G}}\varphi(f)-(-1)^{(k+1)(k'+1)}\varphi(\omega*_{{F}}f),
\]

for $\varphi\in\mathcal{H}om_{\BB^{\sharp}_X}({F},{G}[k, e]),
\omega\in\Omega_X^{k'}(e'),f\in{F}.$

\

\noindent
Then, $\varphi$ is $*$-linear if and only if $\omega*\varphi=0$ for all $\omega\in\BB_X$.

\

\noindent
Another fundamental fact is that we can define a tensor product in $\BB_X$-mod
by defining a $*$-structure in the tensor product in $\BB^{\sharp}_X$-mod,
\[
{F}\otimes_{\BB_X}{G}:=
({F}\otimes_{\BB^{\sharp}_X}{G},*),
\]
where for $\omega\in\Omega_X^{k}(e),\,f\in{F}^{k'}(e'),\,
g\in{G}^{k''}(e'')$,
\[
\omega*(f\otimes g):=\frac{e+e'}{e+e'+e''}(\omega*f)\otimes g
+(-1)^{(k+1)k'}
\frac{e+e''}{e+e'+e''}
f\otimes(\omega*g).
\]

\

\begin{definition}
Let $X$ be a smooth projective variety.
The \emph{category $B_X$-mod} is defined as the category of
equivalence classes of 
bigraded vector spaces 
$$F=\Gamma_{*}({F}) = \bigoplus_{k,e} H^0(X, {F}^{k}(e))$$ 
where
${F}$ is a $\BB_X$-Module.
The equivalence is given by 
$H^0(X,{F}^k(e))= H^0(X,{G}^k(e))$
for $e>>0$ and every $k$.

\

\noindent
A \emph{twisted differential ideal} $I\subseteq B_X$ is a
$B_X$-submodule of $B_X$ in $B_X$-mod.
\[
I=\bigoplus_{k,e} I^{k,e}\subseteq B,\quad 
I^{k,e}\subseteq H^0(X,\Omega^k(e)).
\]
\end{definition}

\

\begin{remark}[Interpretation of the second multiplication] 
We follow \cite[Ch. 2, \S 4D]{MR2394437}.
Let $X$ be a smooth projective algebraic variety and 
let $x\in H^0(X,\OO_X(1))$.
Consider the affine open subset $U=X- Z$, where $Z=\{x=0\}$ 
and $j:U\hookrightarrow X$ is the inclusion.
Then we have a commutative
diagram of sheaves on $X$,
\[
\xymatrix{
\Omega_X^k(e)\ar[r]^{=}\ar[d]^{*x}&\Omega_X^k(e\cdot Z)\ar[d]^d\\
\Omega_X^{k+1}(e+1)\ar[r]^<<<<{=}&\Omega_X^{k+1}((e+1)\cdot Z)
}
\]
The sheaf $\Omega_X^k(e\cdot Z)$ if the sheaf of meromorphic $k$-forms
on $X$ which are regular outside $Z$ and have poles of order $\le e$
along $Z$. That is, if $\omega$ is a section
of $\Omega_X^k(e\cdot Z)$, then $x^e\omega$ is regular. Also, 
the sheaf $j_{*}\Omega^k_U$ is the sheaf of 
of meromorphic $k$-forms
on $X$, regular outside $Z$ and having poles along $Z$ of any order,
$j_{*}\Omega^k_U=\Omega_X^k(*Z)$.
\qed
\end{remark}

\

\begin{theorem}
Let $X$ be a smooth projective variety.
The category $\Omega_X$-dgmod is equivalent to the category
$B_X$-mod,
\[
\Gamma_{*}:\Omega_X\text{-dgmod}\to B_X\text{-mod},\quad 
\bigoplus_{k}{F}^k\mapsto\bigoplus_{e,k} H^0(X,{F}^k(e)).
\]
\end{theorem}
\begin{proof}
Clearly $\Gamma_{*}$ gives an equivalence between 
$\Omega_X^{\sharp}$-mod and $B_X^{\sharp}$-mod. Also, 
it sends $(\Omega_X,d)$ to the $B_X$-Module $B_X$.

\noindent
For ${F}$ be a differential graded $\Omega_X$-Module, $\omega\in H^0(X,\Omega^k_X(e))$
and $f\in H^0(X,{F}^{k'}(e'))$, define
\[
\omega*f:=\begin{cases}
\frac{e}{e+e'}\omega\wedge df-(-1)^k \frac{e'}{e+e'}d\omega\wedge f&
\text{if }e,e'\neq 0,\\
\hfill 0\hfill 
&\text{if }e=0\text{ or }e'=0.
\end{cases}
\]
It is easy to check $\omega*f$ is well defined because
the expression $\omega*f$ is homogeneous in the affine cone of $X$.
Then $\Gamma_{*}({F})$ is a $B_X$-Module.
Assume now that $F$ is $B_X$-Module. Then, there exists
a $\Omega_X^{\sharp}$-Module ${F}$ such that 
$F=\Gamma_{*}({F})$. Let us prove that 
${F}$ is a differential graded $\Omega_X$-Module.
Let $x\in H^0(X,\OO_X(1))$ and let $U=\{x\neq 0\}$,
\[
d\left(\frac{f}{x^e}\right):=(e+1)\frac{x*f}{x^{e+1}},\qquad \frac{f}{x^e}\in
{F}^k(U).
\]
Then $({F},d)$ is in $\Omega_X$-dgmod.
Using the $\mathcal{H}omgr$ construction, it follows 
that $\Gamma_{*}$ sends a map commuting with $d$ to a map
commuting with $x*-$ for all $x\in H^0(X,\OO_X(1))$. 
\end{proof}

\medskip

\subsection{Deformations of twisted exterior differential ideals.}

\

\begin{proposition}\label{tan-ob-tds}
Let $I\subseteq B_X$ be 
a twisted differential ideal over a projective algebraic variety $X$. Then, 
the tangent-obstruction theory associated to deformations of $I$
are 
\[
\Hom_{B_X}(I,B_X/I),\quad
\Ext^1_{B_X}(I,B_X/I).
\]
\end{proposition}

\begin{proof}
Same as before using the explicit computations or the general theory 
of abstract Hilbert schemes.
\end{proof}

\medskip

\begin{theorem}
Let $X$ be a projective algebraic variety and let $\II \in \Hilb(\Omega_X)$. Denote $I=\Gamma_{*}(\II)$.
If $\Ext^1_{B_X}(I, B_X/I)=0$ then $\II$ is a non-singular point of $\Hilb(\Omega_X)$.
\end{theorem}

\begin{proof}
From Proposition \ref{tan-ob-tds},
the deformation functor associated to $I$ has a 
tangent-obstruction theory given
by $\text{Hom}_{B_X}(I,B_X/I)$
and $\Ext^1_{B_X}(I,B_X/I)$.
The result follows from \cite[Cor. 6.2.5]{MR2222646}.
\end{proof}

\newpage

\section{Torsion, depth and saturation.}

\

\noindent
Here we recall some known  facts regarding normal and reflexive coherent Modules.
We shall state the results for a coherent sheaf $G$ on an algebraic variety $X$, following  \cite{MR597077}.
Similar results hold if $G$ is an analytic coherent sheaf on a complex manifold $X$; for this case one may refer to  \cite{okonek1980vector}. 
See also  \cite{bourbaki2007algebre}, Chap. VII, \S 4, and \cite{stacks-project}, Section 15.23: Reflexive modules.

\

\begin{definition}\label{normal} 
Let $X$ be a smooth algebraic variety and $G$  a coherent $\OO_X$-Module.
One says that $G$ is \emph{normal} if for any open $V \subset X$ and any closed subvariety $A \subset V$ of codimension $\ge 2$,
the restriction map $G(V) \to G(V-A)$ is an isomorphism. 
\end{definition}

\begin{proposition}\label{reflexive normal} 
Let $X$ be a smooth algebraic variety and $G$  a coherent  $\OO_X$-Module.
Then $G$ is reflexive if and only if $G$ is torsion-free and normal.
\end{proposition}
\begin{proof}
Same proof as in  \cite{okonek1980vector}, Lemma 1.1.12.
\end{proof}

\

\begin{proposition}\label{torsion} 
Let $X$ be a smooth algebraic variety with $\dim(X) \ge 2$. 
Let $G$ be a coherent $\OO_X$-Module, with singular set $S$, as in Definition \ref{singular scheme}. Assume $\cod(S)\ge 2$.
Then the following conditions are equivalent:
\begin{enumerate}[label=(\arabic*)]
\item\label{h0} $\mathcal{H}^0_S(G)=0$.
\item\label{d1} $\depth_S(G)\geq 1$.
\item\label{s1} $G$ satisfies condition $S_1$ of Serre.
\item\label{tf} $G$ is torsion free.

\end{enumerate}
Assume also that $G$ is torsion-free. Then the following conditions are equivalent:
\begin{enumerate}[label=(\alph*)]
\item\label{H1} $\mathcal{H}^0_S(G)=\mathcal{H}^1_S(G)=0$.
\item\label{d2} $\depth_S (G) \geq 2$.
\item\label{s2} $G$ satisfies condition $S_2$ of Serre.
\item\label{re} $G$ is reflexive.
\item\label{normal} $G$ is normal.
\item\label{in} For any open $V \subset X$ 
 the natural map $G|_V \to \iota_{*}(G|_{V- S})$ is an isomorphism, where $\iota: V- S\to V$ is the inclusion.

\end{enumerate}

\end{proposition}
\begin{proof}  
These equivalences are well-known, see \cite{MR0224620}  and  \cite{MR597077}.
\end{proof}

\medskip

\begin{definition} \label{def saturation}   Let $\II \subset \Omega_X$ be a differential graded ideal.   
We denote   $\bar \II  \subset \Omega_X$ the double dual of $\II$ as an $\OO_X$-Module.
\end{definition}

\begin{proposition} \label{saturation1} With the notation above, $\bar \II$ is a reflexive differential graded ideal, as in Definition  \ref{def pfaff ideal}, e).
\end{proposition}

\begin{proof}
First notice that, by Proposition \ref{torsion}, $\bar \II / \II$ is a torsion $\OO_X$-Module, supported on $S(\II)$. In fact,  $\bar \II / \II$ is the torsion of  $\Omega_X / \II$ as $\OO_X$-Module.
Or, $\bar \II = \{\eta \in \Omega_X/ \exists f \in \OO_X, f \eta \in \II \}$. This implies that $\bar \II$ is a homogeneous ideal of $\Omega_X$. 
To see that $\bar \II$ is stable under exterior derivative, 
let $\eta \in \bar \II(U)$ be a section of $\bar \II$ on an open $U$. As above, there exists $f \in \OO_X(U), f \ne 0,$ such that $f \eta \in \II(U)$. Applying $d$ and multiplying by $f$
we get $f d(f \eta) = f df \wedge \eta + f^2 d\eta \in  \II(U)$. Hence $f^2 d\eta \in \bar \II(U)$, and therefore $d\eta \in \bar \II(U)$, as claimed.
\end{proof}

\medskip

\noindent
The following Proposition allows one to calculate the saturation in some cases.

\begin{proposition} \label{saturation2}  
Let $\II$ be a singular integrable Pfaff ideal such that the singular locus $S(\II)$ has codimension at least two.
Denote $q = \rank (\II^1)$. Then
$$
\bar \II =
\{\theta\in\Omega_X / \ \theta\wedge\bigwedge^q\II^1=0\}.
$$
\end{proposition}

\begin{proof} Let us denote $\widetilde{\II} = \{\theta\in\Omega_X / \ \theta\wedge\bigwedge^q\II^1=0\}$.
It is clear that $\widetilde{\II}$ is a homogeneous ideal of $\Omega_X$. 
Let us see that $\widetilde{\II}$ is reflexive, or equivalently, that $\Omega_X/\widetilde{\II}$ is torsion-free.
For this, suppose ${\theta}\in \Omega_X^r$ and
$f\in\OO_X$ are such that $f\theta\in\widetilde{\II}$, 
that is, $f\theta\wedge\bigwedge^q\II^1=0$. Since $\Omega_X$ is torsion-free 
we have $\theta\wedge\bigwedge^q\II^1=0$, that is,
$\theta\in\widetilde{\II}$, as required.

\

\noindent
Taking double-dual to the inclusion $\II\subseteq \widetilde{\II}$ we get 
$\overline{\II}\subseteq \widetilde{\II}$. Let us see that equality holds.
Denote $Q =  \widetilde{\II} / \overline{\II}$, so we need to prove $Q=0$.
It is clear that $\widetilde{\II}/\II$  is supported on $S(\II)$,
and hence the same is true for $Q$. 
Applying $\mathcal{H}om_{\OO_X}(Q,-)$ to the inclusion $\overline{\II}\subseteq \widetilde{\II}$ we obtain the exact sequence
\[
0\to \mathcal{H}om_{\OO_X}(Q,\overline{\II})\to
 \mathcal{H}om_{\OO_X}(Q,\widetilde{\II})\to
  \mathcal{H}om_{\OO_X}(Q,Q)\to
   \mathcal{E}xt^1_{\OO_X}(Q,\overline{\II})
\]
Since $Q$ is torsion and $\widetilde{\II}$ is torsion-free, 
$\mathcal{H}om_{\OO_X}(Q,\widetilde{\II})=0$.
And since $\overline{\II}$ is reflexive
and $Q$ is supported in codimension at least two, 
$\mathcal{E}xt^1_{\OO_X}(Q,\overline{\II})=0$ (see \cite{bourbaki2007elements}, Proposition 2,  or \cite{grothendieck2005cohomologie}, Proposition (2.4)).
Therefore $\mathcal{H}om_{\OO_X}(Q,Q)=0$  and then
$Q=0$, as claimed.
\end{proof}

\medskip

\begin{definition} \label{saturated Pfaff} A differential graded ideal will be called a \emph{saturated Pfaff ideal} if it is the saturation
of a singular integrable Pfaff ideal.
\end{definition}

\medskip

\begin{corollary}  \label{saturation3}
With notation as in Proposition \ref{saturation2}, we have $\bar \II^r = \Omega^r_X$ for $r > \dim(X)-q$.
\end{corollary}
\begin{proof}
This follows from Proposition \ref{saturation2} and the fact that  $\Omega^r_X = 0$ for $r >  \dim(X)$.
\end{proof}

\medskip

\begin{corollary}   \label{saturation4}
If $\II^1$ is generated freely on the open set $U$ by one-forms $\omega_1, \dots, \omega_q$  then
$$
\bar \II|_U =
\{\theta \in \Omega_U / \ \theta \wedge \omega_1 \wedge \dots \wedge \omega_q=0\}.
$$
That is, $\bar \II|_U$ is the kernel of exterior multiplication by $\omega = \omega_1 \wedge \dots \wedge \omega_q$.
More generally, if $\II^1$ is generated on the open set $U$ by one-forms $\omega_1, \dots, \omega_s$, $s \ge q$,  then
$$
\bar \II|_U =
\{\theta \in \Omega_U / \ \theta \wedge \omega_{i_1} \wedge \dots \wedge \omega_{i_q}=0, \ \forall \ 1 \le i_1 < \dots < i_q \le s \}.
$$

\end{corollary}

\newpage

\section {Calculating $\Hom$ and $\Ext^1$ for differential graded modules.}

\subsubsection{}

In this section we collect several calculations that seem useful for determining deformations and obstructions of exterior differential ideals.

\

\noindent
One objective is to determine  $\mathcal{H}om_{\Omega_X}(\II,\Omega_X/\II)$ for some differential graded ideals $\II$.

\

\noindent
We may rewrite \ref{cicles} in the case where $F = \II$ and $G = \Omega_X/\II$:

\begin{equation}
\xymatrix{
0 \ar[r] &\mathcal{H}om_{\Omega_X}({\II},{\Omega_X/\II})  \ar[r] &\mathcal{H}om_{\Omega_X^{\sharp}}({\II},{\Omega_X/\II})  \ar[r]^d &\mathcal{H}om_{\Omega_X^{\sharp}}({\II},{(\Omega_X/\II)[1]})
}
\label{cicles2}
\end{equation}

\subsubsection{}

Next we prove several propositions on  $\mathcal{H}om$, $\mathcal{E}xt^1$, $\Hom$ and $\Ext^1$ for $\OO_X$, $\Omega^{\sharp}_X$ and $\Omega_X$-Modules, and relations among them.

\

\begin{proposition}\label{hom omega} Let ${F},{G}$ be differential graded $\Omega_X$-Modules, as in Section \ref{preliminaries}. 
Let $S \subset X$ be a (Zariski) closed set and $U = X-S$. We have the exact commutative diagram:

\begin{equation}
{\footnotesize
\xymatrix{
0 \ar[r] & H^0(X, \mathcal{H}om_{\Omega_X}({F},{G}))  \ar[r] \ar[d]^ {a} &H^0(X, \mathcal{H}om_{\Omega_X^{\sharp}}({F},{G}))  \ar[r]^d \ar[d]^{b} &
H^0(X, \mathcal{H}om_{\Omega_X^{\sharp}}({F},{G[1]})) \ar[d] \ar[d]^{c} \\
0 \ar[r] &H^0(U, \mathcal{H}om_{\Omega_X}({F},{G}))  \ar[r] &H^0(U, \mathcal{H}om_{\Omega_X^{\sharp}}({F},{G}))   \ar[r]^d & H^0(U, \mathcal{H}om_{\Omega_X^{\sharp}}({F},{G[1]}))
}
}
\end{equation}
where $a, b, c$ denote restrictions of sections. Then,

\noindent
1) If $b$ is injective then $a$ is injective.

\noindent
2) If $b$ is an isomorphism and $c$ is injective then $a$ is an isomorphism.
\end{proposition}
\begin{proof}
The diagram follows from  \ref{cicles} and the fact that taking sections of a sheaf of abelian groups over an open set is a left-exact functor (\cite{MR2583634}, Ex. (II. 1.8)). 
1) and 2) are obtained by a simple diagram chase.
\end{proof}

\medskip

\begin{proposition}\label{torsion-map}
Let ${F}$ and ${G}$ be differential graded $\Omega_X$-Modules. 

\noindent
For $\lambda \in \Hom_{\Omega_X}({F},{G}[1])$ define
$d_{\lambda}: {G} \oplus{F} \to {G} \oplus{F}$  by  
$$d_{\lambda}(g, f) = (d_{{G}}(g) + \lambda(f), d_{{F}}(f)),$$
for $(g, f)$ a local section of ${G}\oplus{F}$.
Then $({G}\oplus{F}, d_{\lambda})$ is a differential graded $\Omega_X$-Module.
Denote $\tau(\lambda) \in \Ext_{\Omega_X}^1({F},{G})$ the extension of
differential graded  $\Omega_X$-Modules consisting of the trivial extension 
$0 \to G \to {G} \oplus{F} \to F \to 0$ of graded  $\OO_X$-Modules, where 
${G}\oplus{F}$ is endowed with the differential $d_{\lambda}$ as above.

\noindent
The map $\tau: \Hom_{\Omega_X}({F},{G}[1]) \to   \Ext_{\Omega_X}^1({F},{G})$ induces an injective linear map of complex vector spaces 
$$\bar \tau: H^1(\mathcal{H}omgr_{\Omega_X^{\sharp}}({F},{G}),d) \to \Ext_{\Omega_X}^1({F},{G}).$$

\noindent
$\bar \tau$ is called \emph{the torsion map}.
\end{proposition}

\

\begin{proof}
The following proof is motivated by calculations in \cite{MR3640821}.

\

\noindent
Let us first verify that $d_{\lambda}^2 = 0$. In matrix notation, one easily calculates:
\[
\begin{pmatrix}
d_{{G}}&\lambda\\
0&d_{{F}}
\end{pmatrix}^2=
\begin{pmatrix}
0&d_{{G}}\lambda+\lambda d_{{F}}\\
0&0
\end{pmatrix}.
\]
Since $d_{{G}[1]}=-d_{{G}}$ and $\lambda$ commutes with differentials, it follows that $d_{{G}}\lambda+\lambda d_{{F}}=0$.
We obtain that  $({G}\oplus{F}, d_{\lambda})$ is a differential graded $\Omega_X$-Module. 

\

\noindent
Let us show that $\tau$ is $\C$-linear. First, it is easy to check by using the Baer structure of $\Ext^1$ that
\[
\tau(\lambda_1+\lambda_2)=\tau(\lambda_1)+\tau(\lambda_2).
\]
Let us prove $\tau(r\lambda)=r\tau(\lambda)$ for $r\in \C$. As explained in \cite[Exercise A3.26]{MR1322960}, the $\C$-linear structure on $\Ext^1(F,G)$ is
given by the pullback of multiplication by $r:F\to F$. Specifically, given the extension $\tau(\lambda)$, the class of the extension $r\tau(\lambda)$
can be represented as the pullback $E_r$ inside $(G\oplus F)\oplus F$,
\[
\xymatrix{
0\ar[r]& G\ar[r]\ar@{=}[d]& E_r\ar[r]^{\pi_3}\ar[d]& F\ar[d]^{r}\ar[r]& 0\\
0\ar[r]& G\ar[r]& G\oplus F\ar[r]^{\pi}& F\ar[r]& 0
}
\]
where $E_r(U)  = \{(g,rf,f)\,\colon\,g\in G(U), f\in F(U)\}$ and the differential on $E_r$ is the restriction of $d_\lambda\oplus d_F$.
Let us check that $E_r$ is in the same class as the extension $\tau(r\lambda)$. First, define the sheaf isomorphism $c:E_r\to G\oplus F$ given
locally by $c(g,rf,f)=(g,f)$,
\[
\xymatrix{
&&(E_r,d_\lambda\oplus d_F)\ar[dr]\ar[dd]^c\\
0\ar[r]& G\ar[ru]\ar[dr]& & F\ar[r]& 0\\
&&(G\oplus F,d_{r\lambda}) \ar[ur]
}
\]
Clearly $c$ is a morphism of extensions. And it commutes with the differentials:
\begin{align*}
d_{r\lambda}c(g,rf,f) &= d_{r\lambda}(g,f) \\
&= (d_G(g)+r\lambda(f),d_F(f)).\\
c(d_\lambda\oplus d_F)(g,rf,f) &= c(d_G(g)+\lambda(rf),d_F(rf),d_F(f))\\
& = (d_G(g)+\lambda(rf),d_F(f)).
\end{align*}
Then, 
\[
\tau(r\lambda)=r\tau(\lambda),\quad\forall r\in\C.
\]

\

\noindent
To see that $\tau$ induces a map from $H^1(\mathcal{H}omgr_{\Omega_X^{\sharp}}({F},{G}),d)$,
recall that $\Hom_{\Omega_X}({F},{G}[1]) = Z^1(\mathcal{H}omgr_{\Omega_X^{\sharp}}({F},{G}),d)$,
and let us check that $\tau(\lambda) = 0$ if $\lambda$ is a border. Then assume that $\lambda= d(h) = d_{{G}} \circ h-h \circ d_{{F}}$
for some homomorphism $h: F \to G$ of graded $\Omega_X$-Modules. Define $\phi:{G}\oplus {F}\to {G}\oplus {F}$ by
\[
\phi = \begin{pmatrix}
id_{{G}} & h \\
0 & id_{{F}}
\end{pmatrix}.
\]
We claim that $\phi: ({G}\oplus {F}, d_{\lambda}) \to ({G}\oplus {F}, d_{0})$ is an isomorphism of differential graded $\Omega_X$-Modules,
that is, $\phi  \circ  d_{\lambda} = d_{0} \circ \phi$. This equality follows from the calculation:
\begin{align*}
\begin{pmatrix}
id_{{G}}&h\\
0&id_{{F}}
\end{pmatrix}
\begin{pmatrix}
d_{{G}}&\lambda\\
0&d_{{F}}
\end{pmatrix}
-
\begin{pmatrix}
d_{{G}}&0\\
0&d_{{F}}
\end{pmatrix}
\begin{pmatrix}
id_{{G}}&h\\
0&id_{{F}}
\end{pmatrix}
&= \\
\begin{pmatrix}
0 &  \lambda+h \circ d_{{F}} - d_{{G}} \circ h\\
0&0
\end{pmatrix}
&=0.
\end{align*}

\

\noindent
Then, $\tau(\lambda)=\tau(0) \in \Ext_{\Omega_X}^1({F},{G}).$ Since clearly $\tau(0) = 0$, it follows that $\tau(\lambda)=0$.
Therefore $\tau$ factors through the quotient $H^1(\mathcal{H}omgr_{\Omega_X^{\sharp}}({F},{G}),d)$ and defines $\bar \tau$.

\

\noindent
Now let us prove the injectivity of $\bar \tau$. Assume that $\tau(\lambda)=0$.
Then, the extension $\tau(\lambda)$
is equivalent, under a map $\phi$ in $\Omega_X$-dgmod, to the trivial extension $\tau(0)$.
We want to show that $\lambda$ is a border. 
From the following diagram in $\Omega_{X}$-dgmod
we get a characterization of $\phi$,
\[
\xymatrix{
&&{G}\oplus {F}\ar[dr]\ar[dd]^{\phi}&\\
0\ar[r]&{G}\ar[ur]\ar[dr]& &{F}\ar[r]&0\\
&&{G}\oplus {F}\ar[ur]&
}
\]

\

\[
\phi=
\begin{pmatrix}
id_{{G}}&h\\
0&id_{{F}}
\end{pmatrix},\quad
h:{F}\to{G}.
\]

\

\noindent
Given that $\phi$ is a map in $\Omega_X$-dgmod, we know that $\phi  \circ  d_{\lambda} = d_{0} \circ \phi$.
The calculation that was done above then gives
$d_{{G}}h=\lambda+hd_{{F}}$. Therefore $\lambda$ is a border, as we wanted to prove.
\end{proof}

\medskip

\medskip
\noindent
The following Proposition is similar to Proposition (3.7), announced in  \cite{MR3640821}.

\medskip

\begin{proposition} \label{h1 y ext1 2} We keep the notations of Proposition \ref{torsion-map}.
Then we have the following exact sequence of complex vector spaces:

\begin{equation}
\xymatrix{
0 \ar[r] & H^1(\mathcal{H}omgr_{\Omega_X^{\sharp}}({F},{G}),d)  \ar[r]^-{\bar \tau} &  \Ext_{\Omega_X}^1({F},{G})  \ar[r]^-{\alpha} &  \bigoplus_{r \in \Z}  \Ext^1_{\OO_X}({F^r},{G^r}),
}
\label{torsion2}
\end{equation}
where $\bar \tau$ is defined in Proposition  \ref{torsion-map}, and $\alpha$ is the map induced by considering $F$ and $G$ only as graded $\OO_X$-Modules, 
forgetting the rest of the structures of $\Omega_X$-Modules.
\end{proposition}

\begin{proof}
By Proposition \ref{torsion-map}, we only need to check exactness in the middle, that is, $\im \bar \tau = \ker \alpha$.  
For $\lambda \in \mathcal{H}om_{\Omega_X}({F},{G}[1])$, 
by definition $\tau(\lambda) = ({G}\oplus{F}, d_{\lambda})$ is split as  graded $\OO_X$-Module, that is, $\alpha(\tau(\lambda)) = 0$.
Therefore $\alpha \circ \bar \tau = 0$ and $\im \bar \tau \subset \ker \alpha$. 

\noindent
Conversely, to see $\ker \alpha \subset \im \bar \tau$, consider an extension
$$\EE=(0 \to G \to E \to F \to 0) \in  \Ext_{\Omega_X}^1({F},{G}). $$
Suppose that $\alpha(\EE) = 0$, that is, $\EE$ splits as an extension of graded $\OO_X$-Modules. Hence we may assume that $E = G \oplus F$. 
Then the differential must be equal to $d_{\lambda}$
for some $\lambda:{F}\to{G}[1]$. In fact, by Leibnitz rule, 
$\lambda \in \Hom_{\Omega_X^{\sharp}}({F},{G}[1])$.
Finally, from the condition $d^2=0$, we get $\lambda \in 
\Hom_{\Omega_X}({F},{G}[1]) = Z^1(\mathcal{H}omgr_{\Omega_X^{\sharp}}({F},{G}),d)$:
\[
0=d^2=\begin{pmatrix}
d_{{G}}&\lambda\\
0&d_{{F}}
\end{pmatrix}
\begin{pmatrix}
d_{{G}}&\lambda\\
0&d_{{F}}
\end{pmatrix}=
\begin{pmatrix}
0&d_{{G}}\lambda+d_{{F}}\lambda\\
0&0
\end{pmatrix},
\]
and therefore $\EE \in \im(\bar \tau)$, as wanted.
\end{proof}

\

\medskip

\begin{corollary} \label{h1 y ext1 sheaf} With the notations of Propositions \ref{torsion-map} and \ref{h1 y ext1 2},
we have an exact sequence of sheaves of complex vector spaces:

\begin{equation}
\xymatrix{
0 \ar[r] & \HH^1(\mathcal{H}omgr_{\Omega_X^{\sharp}}({F},{G}),d)  \ar[r]^-{\tilde \tau} &  \mathcal{E}xt_{\Omega_X}^1({F},{G})  \ar[r]^-{\tilde \alpha} &  
\bigoplus_{r \in \Z}  \mathcal{E}xt^1_{\OO_X}({F^r},{G^r}),
}
\label{torsion sheaf}
\end{equation}
where $\tilde \tau$ (resp. $\tilde \alpha$) is the associated sheaf map induced by $\bar \tau$ (resp. by $\alpha$).
\end{corollary}

\begin{proof} Following \cite{MR0102537} (4.2.2) let us define the presheaf  $\Ext_{\Omega_X}^1( - ; {F},{G})$ by
$$\Ext_{\Omega_X}^1( - ; {F},{G})(U) = \Ext_{\Omega_U}^1({F}|_U,{G}|_U),$$
for each open $U \subset X$. And similar definition for $\Ext_{\OO_X}^1( - ; {F},{G})$.

\

\noindent
Applying Propositions \ref{torsion-map} and \ref{h1 y ext1 2} on each open  $U \subset X$ we get an exact sequence of presheaves
\begin{equation}
\xymatrix{
0 \ar[r] & B^1 \ar[r] & Z^1  \ar[r]^-{ \tau} &  \Ext_{\Omega_X}^1( - ; {F},{G})  \ar[r]^-{\alpha} &  \bigoplus_{r \in \Z}  \Ext^1_{\OO_X}( - ; {F^r},{G^r}),
}
\label{torsion presheaf}
\end{equation}
where $B^1$ and $Z^1$ are the presheaves of borders and cycles defined by

\

\noindent
$Z^1(U) =  Z^1(\mathcal{H}omgr_{\Omega_U^{\sharp}}({F}|_U,{G}|_U),d)  = \Hom_{\Omega_U}({F}|_U,{G}|_U[1]),$

\

\noindent
$B^1(U) = B^1(\mathcal{H}omgr_{\Omega_U^{\sharp}}({F}|_U,{G}|_U),d) = d(H^0(U, \mathcal{H}omgr_{\Omega_U^{\sharp}}({F}|_U,{G}|_U))).$

\

\noindent
Denoting $\tilde P$ the sheaf associated to a presheaf $P$, we have:

\

\noindent
$\tilde Z^1 / \tilde B^1 = Z^1 / \tilde B^1 = \HH^1(\mathcal{H}omgr_{\Omega_X^{\sharp}}({F},{G}),d)$, 

\

\noindent
$\Ext_{\Omega_X}^1( - ; {F},{G}) \tilde{}  = \mathcal{E}xt_{\Omega_X}^1({F},{G})$, 

\

\noindent
$\Ext^1_{\OO_X}( - ; {F^r},{G^r}) \tilde{}  =  \mathcal{E}xt^1_{\OO_X}({F^r},{G^r})$.
 
 \

\noindent
For the last two equalities see \cite{MR0102537}, paragraph after (4.2.3).

\

\noindent
Applying the exact functor $P \mapsto \tilde P$ to \ref{torsion presheaf} we obtain \ref{torsion sheaf}. 

\end{proof}

\

\medskip

\begin{lemma}\label{homotopies}
If $u,v:{F}\to{G}$ are homotopic
and $u',v':{F}'\to{G}'$  are homotopic
then, 
\[
u'_*u^*,v'_*v^*:\mathcal{H}omgr_{\Omega_X^{\sharp}}({G},{F}')
\to
\mathcal{H}omgr_{\Omega_X^{\sharp}}({F},{G}')
\]
are homotopic. The same is true for
\[
u'\otimes u,v'\otimes v:{F}\otimes_{\Omega_X^{\sharp}}{F}'\to
{G}\otimes_{\Omega_X^{\sharp}}{G}'.
\]
\end{lemma}

\begin{proof}
This proof was extracted from \cite{bourbakihomologicalalgebra}, \S 5.
Let $w,w'$ be such that $u-v=dw+wd$, $u'-v'=dw'+w'd$ and
let us define the homotopy $\omega$ by
$\omega^n=w'_*u^*+(-1)^nv'_*w^*$. 
Let us denote by $\delta$ the differential in 
$\mathcal{H}omgr_{\Omega_X^{\sharp}}({G},{F}')$.
Then, given 
$s\in\mathcal{H}omgr_{\Omega_X^{\sharp}}({G},{F}')_n$,
\[
(\delta\omega+\omega\delta)(s)=
\delta(w'su+(-1)^nv'sw)+\omega(ds-(-1)^n sd)=
\]
\[
d(w'su+(-1)^nv'sw)-(-1)^{n+1}(w'su+(-1)^nv'sw)d+
\]
\[
w'(ds-(-1)^n sd)u+(-1)^{n+1}v'(ds-(-1)^n sd)w=
\]
\[
dw'su
+v'swd
+w'dsu
+v'sdw=
(dw'+w'd)su+v's(wd+dw)
=
\]
\[
(u'-v')su+v's(u-v)=(u'_*u^*-v'_*v^*)s.
\]
For the tensor product a similar proof applies, see \cite{bourbakihomologicalalgebra}, \S 4.
\end{proof}

\medskip

\begin{lemma}\label{hom-depth}
Let $X$ be a smooth algebraic variety and let $S\subseteq X$ be a closed subset.
Let $F, G$ be two coherent $\OO_X$-Modules. Then, 

a) If $\depth_S(G) \ge 2$ then $\depth_S(\mathcal{H}om_{\OO_X}(F,G)) \ge 2$. 

b) If $F$ is locally free then $\depth_S(\mathcal{H}om_{\OO_X}(F,G)) = \depth_S(G).$
\end{lemma}
\begin{proof} 
\
a) See \cite{bourbaki2007elements}, \S 1, Exercise 18.

\noindent
b) It follows from the fact that, locally, $\mathcal{H}om_{\OO_X}(F,G)$
is a direct sum of copies of $G$ and $\mathcal{H}_S^i$ commutes with direct sums.
\end{proof}

\medskip

\begin{proposition}\label{h0-depth}
Let $X$ be a smooth algebraic variety, $S\subseteq X$ a closed subset
and $F$, $G$ differential graded $\Omega_X$-Modules
such that $\depth_S(G) \ge 2$. Then for all $k \in \Z$ we have:
\begin{align*}
& 1) \ \mathcal{H}^0_S(\mathcal{H}omgr_{\Omega_X^{\sharp}}(F,G))=0, \quad
\mathcal{H}^1_S(\mathcal{H}omgr_{\Omega_X^{\sharp}}(F,G))=0, \\
& 2) \ \mathcal{H}^0_S(Z^k)=0, \quad
\mathcal{H}^1_S(Z^k)=0, \\
& 3) \ \mathcal{H}^0_S(B^k)=0,  \quad
\mathcal{H}^1_S(B^k)=0, \\
& 4) \ \mathcal{H}^0_S(H^k)=0, 
\end{align*}
where, as in  \ref{boundaries}, 
$Z^k = Z^k(\mathcal{H}omgr_{\Omega_X^{\sharp}}(F,G),d)$, $B^k = B^k(\mathcal{H}omgr_{\Omega_X^{\sharp}}(F,G),d)$
and $H^k = H^k(\mathcal{H}omgr_{\Omega_X^{\sharp}}(F,G),d)$.

\end{proposition}
\begin{proof}
Recall (\cite{MR0224620}, \cite{grothendieck2005cohomologie} or Proposition \ref{torsion}) 
that for an abelian sheaf $M$, $H^0_S(M) = H^1_S(M) = 0$
is equivalent to the condition that the restriction map $H^0(U, M) \to H^0(U-S, M)$ is bijective, for any open $U \subset X$.
And for $M$ a coherent  sheaf this is also equivalent to $\depth_S(M) \ge 2$.

\

\noindent
From our hypothesis $\depth_S(G) \ge 2$ and Lemma \ref{hom-depth} we have 
$$\depth_S(\mathcal{H}om_{\OO_X}(F,G))\ge 2.$$
Therefore we can extend uniquely any $\OO_X$-linear map 
$\phi \in H^0(U- S,\mathcal{H}om_{\OO_X}(F,G))$
to an $\OO_X$-linear map 
$\phi' \in H^0(U,\mathcal{H}om_{\OO_X}(F,G))$, for any open $U \subset X$.

\

\noindent
Our first three claims are equivalent to the unique extension property for the abelian sheaves
$\mathcal{H}omgr_{\Omega_X^{\sharp}}(F,G)$, 
$Z^k$ and $B^k$. 

\

\noindent
Since the uniqueness of extension holds for $\mathcal{H}om_{\OO_X}(F,G)$, 
it easily follows that it also holds for each one of these abelian subsheaves of $\mathcal{H}om_{\OO_X}(F,G)$.

\

\noindent
Let us check the existence of extension in each case:

\

1) Let $\psi \in H^0(U- S,\mathcal{H}om_{\Omega_X^\sharp}(F,G))$
and $\psi'$ its unique $\OO_X$-linear extension to $U$.
Let $f \in H^0(U,F)$ and $\omega \in H^0(U,\Omega_X)$. Then, 
$\psi'(\omega f)-\omega\psi'(f) \in H^0(U,\mathcal{H}^0_S(G))$.
Since  $\depth_S(G) \ge 2$, we have $\mathcal{H}^0_S(G)=0$, and therefore
$\psi' \in H^0(U,\mathcal{H}om_{\Omega_X^\sharp}(F,G))$.

\

2) Let $\psi \in H^0(U- S,\mathcal{H}om_{\Omega_X}(F,G))$
and $\psi'$ its unique $\OO_X$-linear extension to $U$.
Let $f \in H^0(U,F)$. Then, 
$\psi'(d f)-d\psi'(f) \in H^0(U,\mathcal{H}^0_S(G))$.
Since $\mathcal{H}^0_S(G)=0$, we get
$\psi' \in H^0(U,\mathcal{H}om_{\Omega_X}(F,G))$.

\

3) Let $\psi \in H^0(U - S, B^k)$.
Then $\psi = d(h)$ for some $h \in H^0(U - S,\mathcal{H}om_{\Omega_X^\sharp}(F,G[k-1]))$.
The map $\psi$ is $\OO_X$-linear and satisfies $d(\psi) = 0$.
Hence by 2) we can extend (uniquely) $\psi$ to $\psi' \in H^0(U,\mathcal{H}om_{\Omega_X}(F,G[k]))$. Also, by 1), we can extend (uniquely) $h$ to
$h' \in H^0(U,\mathcal{H}om_{\Omega_X^\sharp}(F,G[k-1]))$.
For all $f \in H^0(U,F)$, $(d(h') - \psi')(f) \in H^0(U,\mathcal{H}^0_S(G))=0$. Then, $\psi'$ is a border.

\

4)  Follows from the exact sequences of local cohomology associated to the
short exact sequences $0 \to B^k \to Z^k \to H^k \to 0$.

\end{proof}

\medskip

\begin{corollary} \label{normal} Let $X$ be a smooth algebraic variety, $S\subseteq X$ a closed subset
and $F$, $G$ differential graded $\Omega_X$-Modules such that $\depth_S(G) \ge 2$. 
Then for any open $U \subset X$ we have a natural commutative diagram 
\begin{equation}
\xymatrix{
& H^0(U, \mathcal{H}om_{\Omega_X}({F},{G}))  \ar[r]^{a} \ar[d] & H^0(U-S, \mathcal{H}om_{\Omega_X}({F},{G})) \ar[d] \\
& H^0(U, \mathcal{H}om_{\Omega_X^{\sharp}}({F},{G}))  \ar[r]^{b}  & H^0(U-S, \mathcal{H}om_{\Omega_X^{\sharp}}({F},{G})) 
}
\label{commutative}
\end{equation}
and the restriction maps $a$ and $b$ are isomorphisms. Also, this diagram is cartesian.
\end{corollary}
\begin{proof}
This diagram is part of the one in Proposition \ref{hom omega}, from which we keep the notation.
The restriction map $b$ is an isomorphism due to Proposition \ref{h0-depth} 1). 
By Proposition \ref{h0-depth} 1) again, $c$ is also an isomorphism. Hence, by Proposition \ref{hom omega}, $a$ is an isomorphism, as claimed.
Since $a$ is an isomorphism, it is clear that the square is cartesian. For later use: this means that if a homomorphism $\varphi: F \to G$ of 
${\Omega_X^{\sharp}}$-Modules commutes with $d$ when restricted to $U$, then it commutes with $d$ in $X$.
\end{proof}

\medskip

\begin{proposition} \label{poincare}
Let $X$ be a smooth algebraic variety. Let $\II$ be a singular integrable Pfaff ideal on $X$, with singular set $S=S(\II)$, see Definition \ref{singular scheme},
and let $G$ be an $\Omega_X^{\sharp}$-Module such that 
\begin{equation}
\II . G =0,  \ G^1  \text{is torsion-free, and} \ G^j  \text {is reflexive for} \ j \ge 2. 
\label{hipotesis G}
\end{equation} 
Then the restriction to the degree-one components 
$$\lambda: \mathcal{H}om_{\Omega_X^{\sharp}}(\II,G) \to \mathcal{H}om_{\OO_X}(\II^1,G^1),$$
$\lambda(\varphi) = \varphi^1$, is an isomorphism of $\OO_X$-Modules.
\end{proposition}
\begin{proof}
Denoting $U=X-S$, let us first see that the restriction 
$$\lambda|_U: \mathcal{H}om_{\Omega_X^{\sharp}}(\II,G)|_U \to  \mathcal{H}om_{\OO_X}(\II^1,G^1)|_U$$
is an isomorphism. 

\

\noindent
We have $\II^0 = 0$ and $\II^k|_U =  \Omega^{k-1}_U  \wedge  \II^1|_U$ for $k > 0$. 
Then it is clear that $\lambda|_U$ is a monomorphism. Here we are using the hypothesis that $G^1$ is torsion-free.

\

\noindent
We claim that $\lambda|_U$ is an isomorphism with inverse 
\begin{equation}
\iota: \mathcal{H}om_{\OO_X}(\II^1,G^1)|_U  \to \mathcal{H}om_{\Omega_X^{\sharp}}(\II,G)|_U, \label{i}
\end{equation}
defined by 
$\iota(s^1)(\mu   \wedge  \omega) =  \mu . s^1(\omega)$, for $s^1 \in \mathcal{H}om_{\OO_X}(\II^1,G^1)|_U$, 
$\omega \in \II^1|_U$ and $\mu \in  \Omega^{k-1}_U$.

\

\noindent
To see this, let us show first that  
$$\im(\lambda|_U) = \{s^1 \in \mathcal{H}om_{\OO_X}(\II^1,G^1)/ \  \tau_1. s^1(\tau_2) +  \tau_2. s^1(\tau_1) = 0, \ \forall \tau_1, \tau_2 \in \II^1\}.$$
Denote $A$ the right hand side. Let us see that $A \subset \im(\lambda|_U)$.
If  $s^1$ belongs to $A$, we show that $s^1$ can be extended to an $\Omega_X^{\sharp}$-linear map $s:\II\to G$.
For each $k \ge 2$ define $\widetilde{s}^k:  \Omega_X^{k-1} \otimes_{\OO_X} \II^1 \to G^{k}$ 
 as $\widetilde{s}^k(\theta \otimes \tau) = \theta . s^1(\tau)$. From the exactness of the Koszul complex \cite{MR1322960}
\[
\xymatrix{
\Omega_X^{k-2}  \otimes_{\OO_X}  S^2(\II^1) \ar[r]^<<<<<{\delta_{2,k-2}}&
\Omega_X^{k-1} \otimes_{\OO_X}  \II^1  \ar[r]^<<<<<{\delta_{1,k-1}}&
\II^{h}\ar[r]&0
}
\]
in order to define $s^k: \II^k \to G^k$ we just need to see that $\widetilde{s}^k(\ker(\delta_{1,k-1}))=0$, or equivalently, 
$\widetilde{s}^k \circ \delta_{2,k-2}=0$:
\[
\widetilde{s}^k \delta_{2,k-2}(\theta \otimes \tau_1 \tau_2 )=
\widetilde{s}^k((\theta \wedge \tau_1)\otimes \tau_2 + (\theta \wedge \tau_2) \otimes\tau_1)=
\]
\[
(\theta \wedge \tau_1) . s^1(\tau_2) + (\theta \wedge \tau_2) . s^1(\tau_1))  = \theta . (\tau_1 . s^1(\tau_2) +  \tau_2 . s^1(\tau_1))  = 0.
\]

\noindent
Then $s: \II \to G$ is well-defined, is $\Omega_X^{\sharp}$-linear and extends $s^1$, as we wanted to show. 

\

\noindent
The reverse inclusion $\im(\lambda|_U) \subset A$ follows easily from the same Koszul complex.
Then we obtain $\im(\lambda|_U) = A$. 

\

\noindent
If $\II . G = 0$ then clearly $A =  \mathcal{H}om_{\OO_X}(\II^1,G^1)$. It follows that 
$\lambda|_U$ is an epimorphism, and hence an isomorphism, as claimed.

\

\noindent
Recall the exact sequence of local cohomology  for a coherent sheaf $F$ and a closed subset $S \subset X$,
\begin{equation}
0\to\mathcal{H}^0_S(F)\to F\to \iota_{*}(F|_{X- S})\to \mathcal{H}^1_S(F)\to0.  \label{local coho}
\end{equation}
From Lemma \ref{hom-depth} a) and \ref{local coho}  we obtain
$$\mathcal{H}om_{\OO_X}(\II^1,G^1) = \iota_{*} (\mathcal{H}om_{\OO_X}(\II^1,G^1)|_U),$$ 
and from Proposition \ref{h0-depth} 1),
\begin{equation}
\mathcal{H}om_{\Omega_X^{\sharp}}(\II,G) = \iota_{*} (\mathcal{H}om_{\Omega_X^{\sharp}}(\II,G)|_U). \label{iso}
\end{equation}
Therefore $\lambda: \mathcal{H}om_{\Omega_X^{\sharp}}(\II,G) \to \mathcal{H}om_{\OO_X}(\II^1,G^1)$ is an isomorphism.
\end{proof}

\medskip

\noindent
The previous Proposition is also true replacing the Pfaff ideal $\II$ by its saturation $\bar \II$, as follows:

\medskip

\begin{proposition} \label{poincare2}
Let $X$ be a smooth algebraic variety. Let $\II$ be a singular integrable Pfaff ideal on $X$, with singular set $S=S(\II)$, and let $\bar \II$ be the saturation of $\II$. 
Let $G$ be an $\Omega_X^{\sharp}$-Module such that $\II . G =0$, $G^1$ is torsion-free and $\depth_S(G^j) \ge 2$ for $j \ge2$, as in \ref{hipotesis G}.  
Then the restriction to the degree-one components 
$$\bar \lambda: \mathcal{H}om_{\Omega_X^{\sharp}}(\bar \II,G) \to \mathcal{H}om_{\OO_X}(\bar \II^1,G^1),$$
$\bar \lambda(\varphi) = \varphi^1$, is an isomorphism of $\OO_X$-Modules. 
\end{proposition}
\begin{proof}
We keep the notations introduced in the proof of Proposition \ref{poincare}.
Again from Proposition \ref{h0-depth} 1) we have,
\begin{equation}
\mathcal{H}om_{\Omega_X^{\sharp}}(\bar \II,G) = \iota_{*} (\mathcal{H}om_{\Omega_X^{\sharp}}(\bar \II,G)|_U). \label{iso2}
\end{equation}
Since $\bar \II|_U = \II|_U$, and using Proposition \ref{poincare}, we have 
$$\iota_{*} (\mathcal{H}om_{\Omega_X^{\sharp}}(\bar \II,G)|_U) =
\iota_{*} (\mathcal{H}om_{\Omega_X^{\sharp}}(\II,G)|_U) = \mathcal{H}om_{\OO_X}(\II^1,G^1),$$ which completes the proof. Notice that the arguments above show that we have 
a commutative diagram of isomorphisms:
\begin{equation}
\xymatrix{
 \mathcal{H}om_{\Omega_X^{\sharp}}(\bar \II,G) \ar[r]^{\bar \lambda} \ar[d]  & \mathcal{H}om_{\OO_X}(\bar \II^1,G^1) \ar[d] \\
 \mathcal{H}om_{\Omega_X^{\sharp}}(\II,G)  \ar[r]^<<<<{\lambda} & \mathcal{H}om_{\OO_X}(\II^1,G^1)
}
\end{equation}

\end{proof}

\medskip

\begin{remark} \label{complex O}

Recall (see \ref{differential}) that 
$\mathcal{H}omgr_{\Omega_X^{\sharp}}({\II},{G}) = \bigoplus_{k \in \mathbb{Z}} \mathcal{H}om_{\Omega_X^{\sharp}}({\II},{G}[k]), \label{homgr}$
with differential  $df =d_{{G}} \circ f -(-1)^k f \circ d_{{\II}}$ for $f \in \mathcal{H}om_{\Omega_X^{\sharp}}({\II},{G}[k])$.

\

\noindent
Let us apply the construction of Proposition \ref{poincare} to $G[k]$ for each $k \in \mathbb{Z}$, to obtain isomorphisms
$\lambda[k]: \mathcal{H}om_{\Omega_X^{\sharp}}(\II,G[k]) \to \mathcal{H}om_{\OO_X}(\II^1,G[k]^1).$

\

\noindent
Hence we have commutative diagrams
\begin{equation}
\xymatrix{
 \mathcal{H}om_{\Omega_X^{\sharp}}(\II,G[k]) \ar[r]^{\lambda[k]} \ar[d]^{d^k}  & \mathcal{H}om_{\OO_X}(\II^1,G^{k+1}) \ar[d]^{d_{\OO}^k} \\
 \mathcal{H}om_{\Omega_X^{\sharp}}(\II,G[k+1])  \ar[r]^<<<<{\lambda[k+1]} & \mathcal{H}om_{\OO_X}(\II^1,G^{k+2})
}
\label{differential2} 
\end{equation}
defining $d_{\OO}^k = \lambda[k+1] \circ d^k \circ \lambda[k]^{-1}$ for $k \in \mathbb{Z}$. More explicitly, for $\varphi^1 \in \mathcal{H}om_{\OO_X}(\II^1,G^{k+1})$ let
$\varphi =  \lambda[k]^{-1} (\varphi^1) \in  \mathcal{H}om_{\Omega_X^{\sharp}}(\II,G[k])$ (see \ref{i}). Then it easily follows that 
\begin{equation}
d_{\OO}^k(\varphi^1) = (-1)^k (d^{k+1}_G \circ \varphi^1 - \varphi^2 \circ d^{1}_{\II}).
\label{differential3} 
\end{equation}

\

\noindent
We thus obtain an isomorphism of complexes of $\OO_X$-Modules
\begin{equation}
\bigoplus_{k\in\mathbb{Z}} \lambda[k]: \mathcal{H}omgr_{\Omega_X^{\sharp}}({\II},{G}) \to \bigoplus_{k\in\mathbb{Z}}  \mathcal{H}om_{\OO_X}(\II^1,G^{k+1}).
\label{complex d0}
\end{equation}
\end{remark}

\medskip

\begin{corollary} \label{tangente} With the notation and hypothesis of Proposition \ref{poincare2} and Remark \ref{complex O} we obtain
\begin{equation}
\mathcal{H}om_{\Omega_X}(\II,G) \cong \ker d_{\OO}^0 = \{ \varphi^1 \in \mathcal{H}om_{\OO_X}(\II^1,G^1) / \ d^{1}_G \circ \varphi^1 - \varphi^2 \circ d^{1}_{\II} = 0 \in G^2\}.
\label{tangente2}
\end{equation}
\end{corollary}

\begin{proof}
From \ref{cicles}, $\mathcal{H}om_{\Omega_X}(\II,G) = \ker d^0$. From Proposition \ref{poincare2} and \ref{differential2}, $\ker d^0 \cong \ker d_{\OO}^0$. Combining with \ref{differential3} the result follows. 
\end{proof}

\medskip

\begin{example} \label{calculo tangente} Suppose $\II$ is locally generated by linearly independent 1-forms $\omega_1, \dots, \omega_q$. As in \ref{frobenius} we have, locally, 
\begin{equation}
d\omega_i  = \sum_j \mu_{ij} \wedge \omega_j,
\label{mu's}
\end{equation} 
for some one-forms $\mu_{ij}$. Let $\varphi$ be a local section of $\mathcal{H}om_{\Omega_X}(\II,G)$.
Then $\varphi$ is determined by its degree-one component $\varphi^1$, which in turn is determined by the images of basis elements $\eta_i = \varphi^1(\omega_i) \in G^1$.
Evaluating the condition of \ref{tangente2} on each $\omega_i$ 
and using  $\varphi^2(d^{1}_{\II}(\omega_i)) = \varphi^2( \sum_j \mu_{ij} \wedge \omega_j) =  \sum_j \mu_{ij} .  \varphi^1(\omega_j)$, we obtain the equivalent condition
\begin{equation}
d^{1}_G(\eta_i) - \sum_{j=1}^q \mu_{ij} . \eta_j = 0 \ \in G^2, \ \forall i = 1, \dots, q.
\label{tangente3}
\end{equation}
Multiplying by $\omega = \omega_1 \wedge \dots \wedge \omega_q$ and after an easy calculation we obtain another equivalent condition for the $\eta_i$, without involving the $\mu_{ij}$'s:
\begin{equation}
d\eta_i  \wedge \omega     +  \sum_{j=1}^q 
(-1)^{q-j}   d\omega_i   \wedge \hat \omega_j   \wedge \eta_j   = 0   \ \in \Omega^{q+2}, \ \forall i = 1, \dots, q.
\label{tangente4}
\end{equation}

\medskip
\noindent
For instance, suppose $\II$ is the saturated Pfaff ideal generated by a single integrable 1-form $\omega \in H^0(\PP^n, \Omega^1_{\PP^n}(e))$ of degree $e$. 
Let us take $G = \Omega/\II$ and assume  $G$ satisfies \ref{hipotesis G}.
Then,
\begin{equation}
\mathcal{H}om_{\Omega_{\PP^n}}(\II,\Omega/\II) \cong  \{\eta \in  \Omega^1_{\PP^n}(e) / \  d\eta \wedge \omega + d\omega \wedge \eta = 0.\}
\label{tangente5}
\end{equation}
\end{example}

\medskip
\noindent
We obtain:

\begin{corollary}
With the notation of Example \ref{calculo tangente}, the first order deformations of the foliation defined by $\omega$ coincide with the first order deformations of the saturated Pfaff ideal generated by $\omega$.
Also, it follows from \ref{tangente4} that the same conclusion holds for the foliation and the saturated Pfaff ideal generated by $\omega_1, \dots, \omega_q$.
\end{corollary}

\medskip

\begin{remark} \label{degree}
Again with the notation of Example \ref{calculo tangente}, suppose $X = \C^{n+1}$ and $\omega_i$ are homogeneous of degree $e_i$. Then in \ref{mu's} we may assume that $\mu_{ij}$ is a homogeneous rational one-form  of degree $e_i - e_j$, and equal to zero if $e_i < e_j$.

\noindent
In fact, if $\bar \mu_{ij}$ is the homogeneous component of $\mu_{ij}$ of degree $e_i - e_j$ then
$d\omega_i  = \sum_j \mu_{ij} \wedge \omega_j = 
\sum_j \bar \mu_{ij} \wedge \omega_j + \sum_j (\mu_{ij} - \bar \mu_{ij}) \wedge \omega_j$.
Since $d\omega_i$ is homogeneous of degree $e_i$ we obtain 
$d\omega_i = \sum_j \bar \mu_{ij} \wedge \omega_j$ and $\sum_j (\mu_{ij} - \bar \mu_{ij}) \wedge \omega_j = 0$, which implies the claim.
\end{remark}

\medskip

\begin{example} \label{formula d0} With notation as in Example \ref{calculo tangente}, we may also write an explicit local expression for $d_{\OO}$, as follows. Let $\varphi^1 \in \mathcal{H}om_{\OO_X}(\II^1,G^{k+1})$. 
Denote $\eta_i = \varphi^1(\omega_i) \in G^{k+1}$. Evaluating \ref{differential3} on $\omega_i$ and combining with \ref{mu's}
we obtain
\begin{equation}
d_{\OO}^k(\varphi^1)(\omega_i) = (-1)^k (d^{k+1}_G(\eta_i) - \sum_j \mu_{ij} .  \eta_j).
\label{differential4} 
\end{equation}
For instance, when $q=1$ and $\II$ is generated by $\omega$, we have 
$$d_{\OO}^k(\varphi^1)(\omega) = (-1)^k (d\eta -  \mu .  \eta),$$
where $d\omega = \mu \wedge \omega$ and $\eta = \varphi^1(\omega)$.
\end{example}

\medskip

\medskip
\noindent
Next, we have a Proposition that will be useful later.

\medskip

\begin{proposition}\label{h1}
Let $\II$ be a singular integrable Pfaff ideal with singular set $S=S(\II)$,
and let $G$ be an $\Omega_X^{\sharp}$-Module such that $\II . G =0$
and $G$ is reflexive as $\OO_X$-Module. Let  $U = X - S$. With notation as in Definition \ref{analytic sheaf}, 
if for some $k \ge 0$ we have $\HH^{k+1}(G^h|_U, d^h)=0$ then 
$$\HH^{k}(\mathcal{H}omgr_{\Omega_X^{\sharp}}(\II,G)^h, d^h)=0.$$
\end{proposition}
\begin{proof}
We have
$\mathcal{H}omgr_{\Omega_X^{\sharp}}(\II, G) = \bigoplus_{k\in\mathbb{Z}} \mathcal{H}om_{\Omega_X^{\sharp}}(\II, G[k])$
with differential defined as in \ref{differential}. The claim amounts to the exactness of the short complex
\begin{equation}
\xymatrix{
\mathcal{H}om_{\Omega_X^{\sharp}}(\II,G[k-1])^h \ar[r]^{d^{k-1}} & \mathcal{H}om_{\Omega_X^{\sharp}}(\II,G[k])^h  \ar[r]^{d^k} & 
\mathcal{H}om_{\Omega_X^{\sharp}}(\II,G[k+1])^h. 
}
\label{short complex1} 
\end{equation}
By Remark \ref{complex O} it suffices to prove the exactness of the isomorphic complex
\begin{equation}
\xymatrix{
\mathcal{H}om_{\OO_X}(\II^1,G^{k})^h \ar[r]^{d_{\OO}^{k-1}} & \mathcal{H}om_{\OO_X}(\II^1,G^{k+1})^h  \ar[r]^{d_{\OO}^k} & \mathcal{H}om_{\OO_X}(\II^1,G^{k+2})^h.
 }
\label{short complex2} 
\end{equation}
This complex may be called a \emph{twisted De Rham} complex of $G$.

\

\noindent
Let us work first in $U = X - S$.

\

\noindent
Let $\varphi^1$ be a local section of $\mathcal{H}om_{\OO_X}(\II^1,G^{k+1})^h$ such that $d_{\OO}^k(\varphi^1) = 0$, that is, by \ref{differential3}, 
$d^{k+1}_G \circ \varphi^1 = \varphi^2 \circ d^{1}_{\II}$. Then for each local section $\omega$ of $\II^1$ such that $d^{1}_{\II}(\omega) = 0$ ($\omega$ is closed) we have
$d^{k+1}_G(\varphi^1(\omega)) = 0$. By hypothesis we know that the complex
\begin{equation}
\xymatrix{
(G^h)^{k} \ar[r]^{d_G^{k}} & (G^h)^{k+1}  \ar[r]^{d_G^{k+1}} & (G^h)^{k+2}
 }
\label{short complex3} 
\end{equation}
is exact. Therefore there exists a local section $\psi_{\omega}$ of $G^{k}$ such that 

\noindent
$\varphi^1(\omega) = d_G^{k}((-1)^{k-1} \psi_{\omega})$. 
Choose a local basis $\omega_1, \dots, \omega_q$ of $\II^1$ such that each $\omega_i$ is closed, in the analytic topology. 
Then there exists a local section $\psi^1$ of
$\mathcal{H}om_{\OO_X}(\II^1,G^{k})^h$ such that $\psi^1(\omega_i) = \psi_{\omega_i}$ for $i = 1, \dots, q$. 
It follows that $\varphi^1 = d_{\OO}^{k-1}(\psi^1)$, which shows that \ref{short complex2} is exact.
Therefore we have $\HH^{k}(\mathcal{H}omgr_{\Omega_X^{\sharp}}(\II,G)^h|_U), d)=0$.

\

\noindent
Let $Z^k$ and $B^k$ be the kernel and the image of $d$ at position $k$. We already know that $B^k|_U = Z^k|_U$. Applying $i_*$ and using the fact that $G$ is reflexive, we get from Proposition  \ref{h0-depth}
\[
B^k = i_*(B^k|_U) = i_* (Z^k|_U)= Z^k.
\]
where $i:U\to X$ is the inclusion. Then $\HH^{k}(\mathcal{H}omgr_{\Omega_X^{\sharp}}(\II,G)^h), d)=0$.  
\end{proof}

\medskip

\begin{remark}  If we assume that in $U = X - S$, $\II$ is locally (in the Zariski topology) generated by exact 1-forms then the result of
Proposition \ref{h1} also holds in the Zariski topology. The proof is the same.
\end{remark}

\medskip

\begin{proposition} \label{lie derivative2}
Let $\II$ be a singular integrable Pfaff ideal with singular set $S=S(\II)$,
and let $G$ be an $\Omega_X^{\sharp}$-Module  as in \ref{hipotesis G}.
Assume also that $\HH^1(G^h, d)=0$. Then we have an epimorphism of sheaves of complex vector spaces in the analytic topology
$$L: \mathcal{H}om_{\OO_X}(\II^1, G^0)^h \to \mathcal{H}om_{\Omega_X}(\II, G)^h$$
such that for  $V \in \mathcal{H}om_{\OO_X}(\II^1, G^0)^h$, 
$$L(V) = d \circ i(V) + i(V) \circ d,$$ 
where $i(V) \in \mathcal{H}om_{\Omega_X^{\sharp}}(\II,G[-1])^h$ extends $V$ as in \ref{i}.

\

\noindent
We obtain the exact sequence of sheaves of complex vector spaces in the analytic topology
\begin{equation}
0 \to \ker L \to  \mathcal{H}om_{\OO_X}(\II^1, G^0)^h \to \mathcal{H}om_{\Omega_X}(\II, G)^h \to 0 \label{sucesion}
\end{equation}
and $\ker L = \{V \in \mathcal{H}om_{\OO_X}(\II^1, G^0)^h / d \circ i(V) + i(V) \circ d=0\} \cong \mathcal{H}om_{\Omega_X}(\II,G[-1])^h$.
\end{proposition}
\begin{proof}
Our hypothesis $\HH^1(G^h, d)=0$ and Proposition \ref{h1} imply that 
$$\HH^{0}(\mathcal{H}omgr_{\Omega_X^{\sharp}}(\II,G)^h, d)=0.$$ 
Therefore, with notation as in \ref{boundaries}, 
$$Z^{0}(\mathcal{H}omgr_{\Omega_X^{\sharp}}(\II,G)^h, d)= B^{0}(\mathcal{H}omgr_{\Omega_X^{\sharp}}(\II,G)^h, d).$$
Recall from \ref{Z^0} that $Z^{0}(\mathcal{H}omgr_{\Omega_X^{\sharp}}(\II,G)^h, d) = \mathcal{H}om_{\Omega_X}(\II, G)^h$. On the other hand, from \ref{boundaries},
$$B^{0}(\mathcal{H}omgr_{\Omega_X^{\sharp}}(\II,G)^h, d) = \im(d: \mathcal{H}om_{\Omega_X^{\sharp}}(\II,G[-1])^h \to \mathcal{H}om_{\Omega_X^{\sharp}}(\II,G))^h.$$
By Proposition \ref{poincare}, $\mathcal{H}om_{\Omega_X^{\sharp}}(\II,G[-1])^h \cong \mathcal{H}om_{\OO_X}(\II^1, G^0)^h$, and hence $L$ is surjective.
\end{proof}

\

\

\noindent
Next we prove a Theorem on  vanishing of $\mathcal{E}xt$ for coherent $\OO_X$-Modules. 
It will be an application of results in \cite{grothendieck2005cohomologie}, which are valid for $X$ a locally noetherian prescheme.

\

\noindent
For related material see   \cite{matsumura1989commutative} (Ischebeck's Theorem), \cite{jinnah1975reflexive} (Lemma 1.1), 
\cite{jothilingam2016projectivity} (Lemma 4), and especially \cite{kleiman1971geometry} (Prop. 2.2.3), and \cite{schlessingerthesis} (Lemma 2, p. 62).

\

\begin{theorem} \label{ext=0}
Let $F$ and $G$ be coherent $\OO_X$-Modules and
let $Y \subset X$ be a closed subset such that $F$ is locally free on $U = X - Y$.
Denote 
$$m = \min\{\depth_Y G -1, \depth_Y \mathcal{H}om_{\OO_X}(F, G) - 2\}.$$
Then 
$$\mathcal{E}xt^j_{\OO_X}(F, G)=0, \  \text{for} \ 1 \le j \le m.$$ 
\end{theorem}
\begin{proof} We consider the sheaves $\mathcal{E}xt^j_{Z}(F, G)$ introduced in Expos\'e VI of \cite{grothendieck2005cohomologie},
where the equivalent notation $\underline{\Ext}^j_{Z}(F, G)$ is used. 

\noindent
For $Z \subset X$ locally closed, $Z' \subset Z$ closed in $Z$, and $Z'' = Z - Z'$, there exists a long exact sequence 
\begin{equation}
\xymatrix{
0 \ar[r] & \mathcal{H}om_{Z'}(F, G)  \ar[r] & \mathcal{H}om_{Z}(F, G)    \ar[r] &\mathcal{H}om_{Z''}(F, G)  \ar[r] &  \dots \\
\ar[r] & \mathcal{E}xt^j_{Z'}(F, G)  \ar[r] &  \mathcal{E}xt^j_{Z}(F, G)   \ar[r] &  \mathcal{E}xt^j_{Z''}(F, G) \ar[r] &  \dots
}
\label{sga1}
\end{equation}
\ref{sga1} is the sheaf version of  \cite{grothendieck2005cohomologie}, VI, Theorem 1.8 and it is obtained by applying the functor  $\mathcal{H}om_{\OO_X}(-, G)$ 
to the exact sequence (1.7.2) as in the deduction of Theorem 1.8.
See also loc. cit., I, Corollaries to Theorem 2.8, or \cite{MR0224620}, Proposition 1.9.

\noindent
Taking $Z = X$, $Z'= Y$ and $Z''= U = X - Y$ we obtain an exact sequence
\begin{equation}
\xymatrix{
0 \ar[r] & \mathcal{H}om_{Y}(F, G)  \ar[r] &  \mathcal{H}om_{\OO_X}(F, G)   \ar[r] & \mathcal{H}om_{U}(F, G) \ar[r] & \dots \\
\ar[r] & \mathcal{E}xt^j_{Y}(F, G)  \ar[r] &  \mathcal{E}xt^j_{\OO_X}(F, G)   \ar[r] &    \mathcal{E}xt^j_{U}(F, G) \ar[r] & \dots
}
\label{sga2}
\end{equation}

\noindent
We claim that 
\begin{equation}
\mathcal{E}xt^j_{U}(F, G) = 0, \  \text{for}  \   1 \le j \le \depth_Y \mathcal{H}om_{\OO_X}(F, G) - 2.
\label{claim}
\end{equation}

\noindent
To see this,  we use the spectral sequence converging to
$\mathcal{E}xt^{p+q}_{U}(F, G)$ with initial terms 
$$\mathcal H^p_U( \mathcal{E}xt^q_{\OO_X}(F, G)),$$
obtained by taking derived functors on the equality of sheaves 
$$\mathcal{H}om_{U}(F, G) = \mathcal H^0_U( \mathcal{H}om_{\OO_X}(F, G)),$$
as in  VI, Theorem 1.6. 

\noindent
By  I, Corollary 2.5, we know that
$$\mathcal H^p_U (\mathcal{E}xt^q_{\OO_X}(F, G)) = R^p \iota_* (\mathcal{E}xt^q_{\OO_X}(F, G)|_U),$$
where $\iota: U \to X$ is the natural inclusion. 

\noindent
Since $F$ and $G$ are coherent and $F|_U$ is locally free, we have
$$\mathcal{E}xt^q_{\OO_X}(F, G)|_U = \mathcal{E}xt^q_{\OO_U}(F|_U, G|_U) = 0,$$ 
for all $q \ge 1$. Therefore $\mathcal H^p_U( \mathcal{E}xt^q_{\OO_X}(F, G)) = 0$ for $q \ge 1$.

\noindent
The remaining initial terms, with $q=0$, are 
$$\mathcal H^p_U( \mathcal{H}om_{\OO_X}(F, G)) = R^p \iota_* (\mathcal{H}om_{\OO_X}(F, G)|_U).$$
By  I, Corollary 2.11, 
$$R^p \iota_* (\mathcal{H}om_{\OO_X}(F, G)|_U) = \mathcal H^{p+1}_Y (\mathcal{H}om_{\OO_X}(F, G)),$$
and by III, Lemma 3.1 (i) and Proposition 3.3 (iv) (or  \cite{MR0224620}, Theorem 3.8), 
$$\mathcal H^{p+1}_Y (\mathcal{H}om_{\OO_X}(F, G)) = 0,$$
if $p+1 <  \depth_Y \mathcal{H}om_{\OO_X}(F, G)$. Thus we obtain \ref{claim}.

\

\noindent
Combining  \ref{claim} with \ref{sga2} we get a five-term exact sequence
\begin{equation}
\xymatrix{
0  \ar[r] & \mathcal{H}om_{Y}(F, G)    \ar[r] & \mathcal{H}om_{\OO_X}(F, G)  \ar[r] & \iota_* \mathcal{H}om_{\OO_U}(F|_U, G|_U)  \\
\ar[r] & \mathcal{E}xt^1_{Y}(F, G)  \ar[r] &  \mathcal{E}xt^1_{\OO_X}(F, G)   \ar[r] & 0,
 }
\label{sgab}
\end{equation}
and isomorphisms
\begin{equation}
 \mathcal{E}xt^j_{Y}(F, G) \cong  \mathcal{E}xt^j_{\OO_X}(F, G),  \  \text{for}  \   2 \le j \le \depth_Y \mathcal{H}om_{\OO_X}(F, G) - 2.
\label{sgac}
\end{equation}

\noindent
Again by \cite{MR0224620}, Theorem 3.8,  we know that   $\mathcal{H}^j_Y(G) = 0$ for $0 \le j < \depth_Y G$. 
Applying VII, Proposition 1.2, $(i) \Leftrightarrow (iv)$, it follows that  $\mathcal{E}xt^j_{Y}(F, G) = 0$ for $0 \le j <  \depth_Y G$. 
Replacing into \ref{sgab} and \ref{sgac} we obtain $\mathcal{E}xt^j_{\OO_X}(F, G) = 0$ for $1 \le j \le m$, as we wanted to prove.
 \end{proof}

\

\begin{corollary} \label{cor ext=0}
With the notation of Theorem \ref{ext=0}, if 
$$\depth_Y G \ge 2 \ \text{and} \ \depth_Y \mathcal{H}om_{\OO_X}(F, G) \ge 3,$$ 
then
$\mathcal{E}xt^1_{\OO_X}(F, G)=0.$ 
\end{corollary}

\newpage

\section{Irreducible components of moduli spaces.}

\

\subsection{On reduced irreducible components.}

 \
 
 \noindent
In this subsection we state a general Proposition which is useful  to determine reduced irreducible components of moduli spaces. This was used in 
some other works, like \cite{cukierman2019stability}, \cite{cukierman2008stability}, \cite{cukierman2009stability}.

\

\begin{proposition}\label{compo}
Let $X$ be a reduced and irreducible scheme.
Let $f:X\to Y$ be a map,
let $x\in X$ be a non-singular point and let $y=f(x)\in Y$.

If $df(x):TX(x)\to TY(y)$
is surjective, then $\overline{f(X)}$ is a reduced and irreducible
component of $Y$.
\end{proposition}

\begin{proof}
Since a ring map sends nilpotents to nilpotents, $f$ factorizes through $Y_{red}$ . Then,
\[
df(x) (TX(x))\subseteq TY_{red}(y)\subseteq TY(y).
\]
Using the hypothesis, this implies that $y$ is a reduced point and that 
the dimension of $\overline{f(X)}$ coincides with
the dimension of $Y$ at $y$.
\end{proof}

\

\medskip

\subsection{Moduli maps for families of exterior differential ideals.} \label{moduli maps}

\

\noindent
Let $f: X  \dashrightarrow Y$ be a rational map. We shall denote $f(X)$ the closure of the image of the rational map $f$, that is, the Zariski closure of $f(U)$ in $Y$, where $U \subset X$ is any dense Zariski open such that $f$ is a morphism on $U$. 
 
 \
 
 \noindent
Let $f: \XX \to S$ be a morphism of schemes and let $\II \subset \Omega_{\XX/S}$ be a family of differential graded ideals parametrized by $S$ as in Definition \ref{def family}.
When $\XX = X \times S$ and $f: X \times S \to S$  is the trivial family ($f$ is the natural projection)  denote 
$$\mu_{\II}: S \to \Hilb(\Omega_X)$$ 
the moduli map induced by $\II$, according to Definition \ref{def hilb}.

\

\noindent
Assume $S$ is reduced and irreducible. If $\II \subset \Omega_{X \times S/S}$ is a coherent sheaf of differential graded ideals, where $\Omega_{X \times S/S}/\II$ is not necessarily flat over $S$, by generic flatness (see \cite{mumford1966lectures}) there exists a dense Zariski open $U \subset S$ such that $(\Omega_{X \times S/S}/\II)|_{X \times U}$ is flat over $U$. Therefore we have a morphism 
$\mu_{\II|_{X \times U}}: U \to \Hilb(\Omega_X)$, and hence a rational map
$$\mu_{\II}: S \dashrightarrow \Hilb(\Omega_X).$$
We denote, as above, $\mu_{\II}(S) \subset \Hilb(\Omega_X)$ the closure of the image of this rational map.

\

\noindent
Also, if $\II \subset \Omega_{X \times S/S}$ is again a coherent sheaf of differential graded ideals, consider its double-dual or saturation 
as $\OO_{X \times S}$-Module $\bar \II \subset \Omega_{X \times S/S}$. 
Then $\bar \II$ is a coherent sheaf of differential graded ideals, as in Proposition \ref{saturation1}. Hence we have a rational moduli map
$$\mu_{\bar \II}: S \dashrightarrow \Hilb(\Omega_X).$$

\

\newpage

\section{Stability of sums of differential ideals.} \label{section pfaff ideals}   

\

\subsection{First theorem on stability.}

\

\noindent
Let $\II = \langle\omega_1, \dots, \omega_q\rangle$ be a singular integrable Pfaff ideal generated by twisted one-forms $\omega_1, \dots, \omega_q$.
The Frobenius condition \ref{frobenius} holds, in particular, if each $\omega_j$ satisfies it individually, that is, $ \omega_j \wedge d\omega_j  = 0$ for all $j$. 
In this case, $\II = \langle\omega_1, \dots, \omega_q\rangle = \sum_{j=1}^q \II_j$
where each $\II_j = \langle\omega_j\rangle$ is a singular integrable Pfaff ideal \ref{pfaff ideals}.
We shall relate the deformations of $\II$ with the deformations of the individual $\II_j$'s. This will allow us to obtain stability results for singular integrable Pfaff ideals.

\

\noindent
More generally, let us fix $q$ and study differential graded ideals $\II \subset \Omega_X$ of the form 
$$\II =  \sum_{j=1}^q \II_j$$ 
where each $\II_j \subset \Omega_X$ is a differential graded ideal. We intend to relate the deformations of $\II$ with the deformations of the $\II_j$'s.

\

\begin{remark}
If each $\II_j$ is a singular integrable Pfaff ideal then $\II =  \sum_{j=1}^q \II_j$ is a singular integrable Pfaff ideal, and its matrix of one-forms $(\eta_{ij})$ as in \ref{frobenius} has a block structure, each block corresponding to each $\II_j$. In case each $\II_j = \langle\omega_j\rangle$ is generated by an integrable one-form $\omega_j$, the matrix $(\eta_{ij})$ for 
$\II = \langle\omega_1, \dots, \omega_q\rangle$ is diagonal.
\end{remark}
 
 \
 
 \noindent
 Let us fix  an integer $q$ such that $1 \le q \le n = \dim X$. For $j = 1, \dots, q$ let $\mathcal{F}_j \subset  \Hilb(\Omega_X)$ be an irreducible component.
 We shall assume that each  $\mathcal{F}_j$  is \emph{generically reduced}.
 Denote $\tilde \II_j$ the corresponding universal differential graded ideal on $X \times \mathcal{F}_j$. 
 Let $\pi_j: X \times \prod_i \mathcal{F}_i \to X \times \mathcal{F}_j$ be the canonical projection and consider $\tilde \II = \sum_j \pi_j^*(\tilde \II_j)$. 
 Then $\tilde \II$ is a coherent sheaf of differential graded ideals on $X \times \prod_i \mathcal{F}_i$ and as in \ref{moduli maps} we have a rational map
  
\begin{equation}
 \mu_{\tilde \II}:  \mathcal{F}_1 \times \dots \times \mathcal{F}_q  \dashrightarrow  \Hilb(\Omega_X)   \label{sigma}
\end{equation}
such that $\mu_{\tilde \II}(\II_1, \dots, \II_q)  = \sum_{j=1}^q \II_j$, for $\II_j$ a general point of $\mathcal{F}_j$. Notice that each $\Omega/\tilde \II_j$ is flat over $\mathcal{F}_j$, 
but $\Omega/\tilde \II$ is not necessarily flat over  $\prod_i \mathcal{F}_i$; for this reason  $\mu_{\tilde \II}$ is a rational map, but not necessarily a regular map.

 \
 
 \noindent
We shall denote
\begin{equation}
\langle\mathcal{F}_1, \dots, \mathcal{F}_q \rangle \ = \  \mu_{\tilde \II}(\mathcal{F}_1 \times \dots \times \mathcal{F}_q), \label{sigma2}
\end{equation}
the closure of the image of the rational map  $\mu_{\tilde \II}$, that is, the collection of all the differential graded ideals of the form  
$\sum_{j=1}^q \II_j$ 
where $(\II_1, \dots, \II_q)$ is a general point of  $\mathcal{F}_1 \times \dots \times \mathcal{F}_q$. 

\

\noindent
We assume that each $\mathcal{F}_j$ is a projective variety. 
Then, since  $\langle\mathcal{F}_1, \dots, \mathcal{F}_q\rangle$ is the image of an irreducible projective variety, it is an irreducible subvariety of $\Hilb(\Omega_X)$.

\
 
 \noindent
Our main results, Theorem (A)  \ref{theorem pfaff}  and Theorem (B)  \ref{theorem pfaff2}, will give conditions under which $\langle\mathcal{F}_1, \dots, \mathcal{F}_q\rangle$ is a reduced irreducible component of 
$\Hilb(\Omega_X)$. Before stating them, we need to introduce some more notation.

\

\noindent
Let
\begin{equation}
s: \bigoplus_i \II_i \to  \Omega_X  \label{sum}
\end{equation}
denote the sum map $s(\omega_1, \dots, \omega_q) = \sum_i \omega_i$ for $\omega_i \in \II_i$. It is a homomorphism of $\Omega_X$-Modules with image $\II = \sum_i \II_i$. 

\

\noindent
Denote
\begin{equation}
R = \ker s = \left\{(\omega_1, \dots, \omega_q)\ \colon \ \omega_i \in \II_i,  
\ \sum_i \omega_i = 0\right\}.  \label{R2} 
\end{equation}
One has the exact sequence of $\Omega_X$-Modules
\begin{equation}
\xymatrix{
 0 \ar[r]   &R \ar[r]^-{\iota} &\bigoplus_i \II_i \ar[r]^-s &\II \ar[r] &0,  
 }
 \label{R} 
\end{equation}
where $\iota$ denotes the natural inclusion.

\

\noindent
Applying the functor $F =  \mathcal{H}om_{\Omega_X}( - , \Omega_X/\II)$ to \ref{R} we obtain the exact sequence 
\begin{equation}\label{diagram2}
\small{
\xymatrix{
0 \ar[r] &\mathcal{H}om_{\Omega_X}(\II,\Omega_X/\II)   \ar[r]^-a &\mathcal{H}om_{\Omega_X}(\bigoplus_i \II_i, \Omega_X/\II)    \ar[r]^-r   &\mathcal{H}om_{\Omega_X}(R,\Omega_X/\II) \\
 \ar[r]   &\mathcal{E}xt^1_{\Omega_X}(\II,\Omega_X/\II)   \ar[r] &\mathcal{E}xt^1_{\Omega_X}(\bigoplus_i \II_i, \Omega_X/\II)    \ar[r]   &\mathcal{E}xt^1_{\Omega_X}(R,\Omega_X/\II), 
 }   
 }
\end{equation}
where $a = F(s)$ and $r = F(\iota)$. 

\

\noindent
Now we may state one of our main results.

\

\begin{theorem}[A] \label{theorem pfaff} Let $\mathcal{F}_j$ be an irreducible component of  $\Hilb(\Omega_X)$, for $j= 1, \dots, q$.
For $\II_j \in \mathcal{F}_j$ denote $\II = \sum_{j=1}^q \II_j$. Let us assume:

\

\noindent
a)  $\Hilb(\Omega_X)$ is reduced at a general point of  $\mathcal{F}_j$, for all $j= 1, \dots, q$.

\

\noindent
b) The natural maps $\Hom_{\Omega_X}({\II_j},{\Omega_X/\II_j}) \to \Hom_{\Omega_X}({\II_j},{\Omega_X/\II})$ are surjective, 
for $\II_j \in \mathcal{F}_j$ a general point and for all $j= 1, \dots, q$.

\

\noindent
Then the derivative of  $\mu_{\tilde \II}$  (\ref{sigma})  at a general point is surjective. 
Therefore, $\langle\mathcal{F}_1, \dots, \mathcal{F}_q\rangle$ (\ref{sigma2}) is an irreducible component of $\Hilb(\Omega_X)$. 
Also, $\Hilb(\Omega_X)$ is reduced at a general point of $\langle\mathcal{F}_1, \dots, \mathcal{F}_q\rangle$.
\end{theorem}

\begin{proof}
For $(\II_1, \dots, \II_q) \in \mathcal{F}_1 \times \dots \times \mathcal{F}_q$ a general point, let us analyse the derivative of $\mu_{\tilde \II}$ at $(\II_1, \dots, \II_q)$:
$$d \mu_{\tilde \II} (\II_1, \dots, \II_q): \bigoplus_{i=1}^q \Hom_{\Omega_X}(\II_i,\Omega_X/\II_i) \to  \Hom_{\Omega_X}(\II,\Omega_X/\II).$$
Here we used the description of the tangent space of $\Hilb(\Omega_X)$ from Corollary \ref{tangent obstruction}.

 \
 
 \noindent
 To simplify the notation, let us write $d \mu_{\tilde \II} (\II_1, \dots, \II_q) = d \mu$, when we fix $\II_i \in \mathcal{F}_i$.

\

\noindent
Our main goal is to show that $d \mu$ is surjective.

\

\noindent
This surjectivity, combined with hypothesis a) and Proposition \ref{compo} readily imply the next claim: $\langle\mathcal{F}_1, \dots, \mathcal{F}_q\rangle$ 
is an irreducible component of $\Hilb(\Omega_X)$ and $\Hilb(\Omega_X)$ is reduced at the general point of this component.

\

\noindent
Let us make explicit the derivative $d \mu$. 

\

\noindent
Denote $\pi_i: \Omega_X/\II_i \to \Omega_X/\II$ the natural projections. We have exact sequences of $\Omega_X$-Modules
\begin{equation}
 0 \to \II/\II_i \to  \Omega_X/\II_i \to \Omega_X/\II \to 0 .  \label{proj}
\end{equation}
For $i = 1, \dots, q$ let $\varphi_i \in \Hom_{\Omega_X}(\II_i,\Omega_X/\II_i)$ represent a first-order deformation of $\II_i$ as in Corollary \ref{tangent obstruction}.
We claim that
\begin{equation}
d\mu(\varphi_1, \dots, \varphi_q)( \sum_{1=1}^q \omega_i) = \sum_{1=1}^q \pi_i  \varphi_i(\omega_i)   \label{derivative}
\end{equation}
for $\omega_i \in \II_i$. That is, we have a commutative diagram
\begin{equation}
\xymatrix{
 0 \ar[r] & R \ar[r]  & \bigoplus_i \II_i    \ar[d]_{\oplus_i \varphi_i}   \ar[r]^s  & \II \ar[d]^{d \mu(\varphi_1, \dots, \varphi_q)}   \ar[r]  & 0 \\
      & & \bigoplus_i \Omega_X/\II_i \ar[r]^{\pi}  & \Omega_X/\II 
}  
\label{diagram1}
\end{equation}
where $\pi = \sum_i \pi_i$. 

 \

\noindent 
To see \ref{derivative},  let $\II_{i  {\epsilon}} \subset \Omega_{\epsilon}$ be a first order deformation of $\II_i$, for $i = 1, \dots, q$.
Then $(\II_{1  {\epsilon}}, \dots, \II_{q  {\epsilon}})$ is a tangent vector (or $\C[\epsilon]$-valued point) of $\mathcal{F}_1 \times \dots \times \mathcal{F}_q$ 
at the point $(\II_1, \dots, \II_q)$. Thus, the derivative of $\mu_{\tilde \II}$ at the point $(\II_1, \dots, \II_q)$ in the direction of the vector $(\II_{1  {\epsilon}}, \dots, \II_{q  {\epsilon}})$ is
$$d \mu_{\tilde \II}(\II_1, \dots, \II_q) (\II_{1  {\epsilon}}, \dots, \II_{q  {\epsilon}}) =    \mu_{\tilde \II} (\II_{1  {\epsilon}}, \dots, \II_{q  {\epsilon}}) = 
\sum_{i=1}^q \II_{i  {\epsilon}}.$$
Denote $\II_{{\epsilon}} = \sum_i \II_{i  {\epsilon}}$, which is a first order deformation of $\II = \sum_i \II_{i}$.
Following Proposition \ref{first order deformations}, $\II_{i  {\epsilon}}$ corresponds to $\varphi_i \in \Hom_{\Omega_X}(\II_i,\Omega_X/\II_i)$
and $\II_{\epsilon}$ corresponds to a certain $\varphi \in \Hom_{\Omega_X}(\II,\Omega_X/\II)$. This means that
$\II_{i  {\epsilon}} = \II_i \times_{\varphi_i} \Omega $  and  $\II_{\epsilon} = \II\times_\varphi \Omega$. Thus, we have
$\varphi = d\mu(\varphi_1, \dots, \varphi_q)$ and $\sum_i \ \II_i \times_{\varphi_i} \Omega = \II \times_\varphi \Omega$. 
The inclusion $\II_i \times_{\varphi_i} \Omega \subset \II \times_\varphi \Omega$ easily implies that 
$\varphi(\omega_i) = \pi_i \varphi_i(\omega_i)$ for $\omega_i \in \II_i$. 
Hence $\varphi(\sum_i \omega_i) = \sum_i \pi_i \varphi_i(\omega_i)$ for $\omega_i \in \II_i$, as claimed.

\

\noindent
From \ref{diagram2} and \ref{diagram1} we obtain the commutative diagram
 
 \begin{equation}\label{diagram3}
\small{
\xymatrix{
0 \ar[r] &\Hom_{\Omega_X}(\II,\Omega_X/\II)   \ar[r]^-a & \bigoplus_i \Hom_{\Omega_X}(\II_i, \Omega_X/\II)    \ar[r]^-r   &\Hom_{\Omega_X}(R,\Omega_X/\II) \\
  & & \ar[lu]^{\ d\mu}  \bigoplus_i \Hom_{\Omega_X}(\II_i,\Omega_X/\II_i)    \ar[u]^{a'}
 }   
 }
\end{equation}
where $a'$ is defined as $a'=\oplus_i a'_i$ and the $a'_i$ are obtained from $\pi_i: \Omega_X/\II_i \to \Omega_X/\II$ by applying $\Hom_{\Omega_X}(\II_i, - )$. 

\noindent
Our hypothesis b) says that each $a'_i$ is surjective. Hence $a'$ is surjective. And it follows from \ref{diagram3} that $d \mu$ is surjective, as claimed.
\end{proof}

\medskip
 
\begin{corollary} \label{dimension of components}
In the conditions of Theorem (A)  \ref{theorem pfaff}, we may calculate the dimension of the component  $\langle\mathcal{F}_1, \dots, \mathcal{F}_q\rangle$ as follows:
\begin{eqnarray}
\dim \langle\mathcal{F}_1, \dots, \mathcal{F}_q\rangle = \dim \Hom_{\Omega_X}(\II,\Omega_X/\II)   \nonumber \\
 = \sum_j \dim  \Hom_{\Omega_X}({\II_j},{\Omega_X/\II_j})   -  \sum_j \dim \Hom_{\Omega_X}({\II_j},{\II/\II_j})   \nonumber \\
 = \sum_j \dim  \mathcal{F}_j   -  \sum_j \dim \Hom_{\Omega_X}({\II_j},{\II/\II_j}).  \nonumber
 \end{eqnarray} 
 where $\II_j \in \mathcal{F}_j$ is a general point for all $j$.
\end{corollary}
\begin{proof}
From \ref{diagram3} and  \ref{longexact} we have 
$$\ker d\mu = \ker a' = \bigoplus_j \Hom_{\Omega_X}({\II_j},{\II/\II_j}).$$
Since $d \mu$ is surjective the stated equalities of dimensions follow from Theorem (A)  \ref{theorem pfaff}.
\end{proof}

\medskip
\noindent
Let's point out that $\Hom_{\Omega_X}({\II_j},{\II/\II_j})$ parametrizes first-order deformations of $\II_j$ that stay contained in $\II$. 

\medskip

\subsection{Some remarks on the obstructions to stability of sums.}  

\

\begin{remark} \label{remarkpfaff} Applying $\Hom_{\Omega_X}({\II_j},{-})$ to \ref{proj}  we obtain the exact sequence
\begin{equation}
\xymatrix{
0 \ar[r]  &\Hom_{\Omega_X}({\II_j},{\II/\II_j}) \ar[r]  &\Hom_{\Omega_X}({\II_j},{\Omega_X/\II_j}) \ar[r]^{a'_j}  &\Hom_{\Omega_X}({\II_j},{\Omega_X/\II})  \\
\ar[r]  &\Ext^1_{\Omega_X}({\II_j},{\II/\II_j})  \ar[r]  &\Ext^1_{\Omega_X}({\II_j},{\Omega_X/\II_j})  \ar[r]  &\Ext^1_{\Omega_X}({\II_j},{\Omega_X/\II}) 
}
\label{longexact}
\end{equation}
Then hypothesis b) in Theorem (A)  \ref{theorem pfaff}  is equivalent to the injectivity of 
\begin{equation}
\Ext^1_{\Omega_X}({\II_j},{\II/\II_j}) \to \Ext^1_{\Omega_X}({\II_j},{\Omega_X/\II_j}), \ \forall j.  \label{condition0}
\end{equation}
Taking into account Proposition \ref{local-global} it would be enough to prove
\begin{equation}
H^0(X,\mathcal{E}xt_{\Omega_X}^{1}({\II_j},{\II/\II_j})) = 0   \ \  \text{and} \ \   H^1(X,\mathcal{H}om_{\Omega_X}({\II_j},{\II/\II_j})) = 0. \label{condition1}
\end{equation}
\end{remark}

\

\noindent 
\begin{remark} \label{strategy}
In order to study  \ref{condition0} it will be necessary for us to work also in the analytic setting, see Definition \ref{analytic sheaf}. We show  in Proposition \ref{comparison2} below  that the surjectivity in b) 
of Theorem (A)  \ref{theorem pfaff}  is equivalent to the similar surjectivity in the analytic topology. As in Remark \ref{remarkpfaff}, 
this surjectivity is equivalent to the injectivity of
\begin{equation}
\Ext^1_{\Omega^h_X}({\II_j}^h, {(\II/\II_j})^h) \to \Ext^1_{\Omega^h_X}({\II_j}^h,({\Omega_X/{\II_j}})^h), \ \forall j . \label{condition4}
\end{equation}

\noindent
We will deduce the injectivity in \ref{condition4} by proving 
\begin{equation}
H^0(X, \mathcal{E}xt^1_{\Omega^h_X}({\II_j}^h, ({\II/\II_j})^h)) = 0,
\label{1obstruction}
\end{equation} 
in Proposition \ref{cumbersome} 
(in fact we will prove $\mathcal{E}xt^1_{\Omega^h_X}({\II_j}^h, ({\II/\II_j})^h) = 0$),
and 
\begin{equation}
H^1(X,\mathcal{H}om_{\Omega^h_X}({\II_j}^h, ({\II/\II_j})^h)=0,
\label{2obstruction}
\end{equation}
in Proposition \ref{vanishing}.

\end{remark}

\

\medskip
\noindent
We need the following Proposition from \cite{deligne1969equation}, Chapter II, (6.6.1), which is an elementary case of the Grothendieck-Deligne Zariski/analytic comparison theorems. We recall the statement for convenience:

\begin{proposition} \label{comparison}
Let $X$ be a proper algebraic variety over the complex numbers, and let $C$ be a differential graded coherent $\OO_X$-Module. 
Then the natural map of finite-dimensional complex vector spaces of hypercohomology
$$\mathbb{H}^*(X, C) \to \mathbb{H}^*(X^h, C^h),$$
is an isomorphism.
\end{proposition}
\begin{proof}
The functor $h$ induces a homomorphism of spectral sequences of hypercohomology
$$H^q(X, C^p) \Rightarrow \mathbb{H}^*(X, C),$$  
$$H^q(X^h, (C^p)^h) \Rightarrow \mathbb{H}^*(X^h, C^h).$$
By GAGA, since $X$ is proper, the morphism in the first page 
$$H^q(X, C^p) \to H^q(X^h, (C^p)^h),$$
is an isomorphism.
Hence the limit $\mathbb{H}^*(X, C) \to \mathbb{H}^*(X^h, C^h)$ is also an isomorphism.
\end{proof}

\

\noindent
On a similar vein, we have the following:

\begin{proposition} \label{comparison2}
Let $X$ be a proper algebraic variety over the complex numbers and let $F, G, H$ be differential graded coherent $\OO_X$-Modules. Then,

\noindent
a) The natural map 
$$H^0(X, \mathcal{H}om_{\Omega_X}({F},{G})) \to H^0(X, \mathcal{H}om_{\Omega^h_X}({F^h},{G^h}))$$ 
is an isomorphism.

\noindent
b) For any morphism $G \to H$ we have an induced commutative square where the vertical maps are isomorphisms,
\begin{equation}
\xymatrix{
H^0(X, \mathcal{H}om_{\Omega_X}({F},{G}))  \ar[r]^a \ar[d] & H^0(X, \mathcal{H}om_{\Omega_X}({F},{H}))  \ar[d]  \\
H^0(X^h, \mathcal{H}om_{\Omega^h_X}({F^h},{G^h}))  \ar[r]^b  & H^0(X^h, \mathcal{H}om_{\Omega^h_X}({F^h},{H^h}))
}
\label{square}
\end{equation}
Therefore, $a$ is surjective if and only if $b$ is surjective. 
In particular, with notation as in Theorem (A)  \ref{theorem pfaff},
the natural map $\Hom_{\Omega_X}({\II_j},{\Omega_X/\II_j}) \to \Hom_{\Omega_X}({\II_j},{\Omega_X/\II})$ is surjective if and only if
the similar map in the analytic topology 
$$\Hom_{\Omega^h_X}({\II_j}^h,({\Omega_X/\II_j})^h) \to \Hom_{\Omega^h_X}({\II_j}^h,({\Omega_X/\II})^h)$$ 
is surjective.

\end{proposition}

\begin{proof} 
a) If $C$ is a differential graded coherent $\OO_X$-Module, then by GAGA the natural map  $H^0(X, C) \to H^0(X^h, C^h)$ is an isomorphism.
Taking $C = \mathcal{H}omgr_{{\Omega_X}^{\sharp}}({F},{G})$, see \ref{cicles},  we obtain the following commutative square where the vertical maps are isomorphisms:
\begin{equation}
\xymatrix{
H^0(X, \mathcal{H}om_{{\Omega_X}^{\sharp}}({F},{G}))  \ar[r]^{d} \ar[d] & H^0(X, \mathcal{H}om_{{\Omega_X}^{\sharp}}({F},{G[1]})) \ar[d] \\
H^0(X^h, \mathcal{H}om_{{\Omega_X^h}^{\sharp}}({F^h},{G^h}))  \ar[r]^{d^h}  & H^0(X^h, \mathcal{H}om_{{\Omega^h_X}^{\sharp}}({F^h},{G^h[1]}))
}
\end{equation}
It follows that $\ker(d) = H^0(X, \mathcal{H}om_{\Omega_X}({F},{G})) \to \ker(d^h) = H^0(X^h, \mathcal{H}om_{\Omega^h_X}({F^h},{G^h}))$ is an isomorphism, as claimed.

\noindent
b) By a) the vertical maps in \ref{square} are isomorphisms. The statement on hypothesis b) of Theorem  \ref{theorem pfaff} follows immediately.
\end{proof}

\

\begin{corollary}
Let $X$ be a proper algebraic variety over the complex numbers and let $\II \subset \Omega_X$ be a differential graded ideal. Then the natural map
$$H^0(X, \mathcal{H}om_{\Omega_X}({\II},{\Omega_X/\II})) \to H^0(X, \mathcal{H}om_{\Omega^h_X}({\II^h},{(\Omega_X/\II)^h}))$$ 
is an isomorphism. That is, the first order deformations of $\II$ and of $\II^h$ are the same.
\end{corollary}
\begin{proof}
If follows from Proposition \ref{comparison2} a)  taking $F=\II$ and $G=\Omega_X/\II$.
\end{proof}

\

\medskip

\subsection{Vanishing of the first obstruction.}  

\

\noindent
We keep the notations above. In particular,  for a sequence of  ideals 
$\II_1,\ldots,\II_q$ on the smooth projective variety $X$, let $\II=\II_1+\ldots+\II_q$ and $S=S(\II)$ its singular set. 

\

\noindent
For simplicity of notation it will be convenient sometimes to replace $\II_j \subset \II$ by singular integrable Pfaff ideals $\JJ$ and $\KK$ such that $\JJ \subset  \KK$. 

\

\noindent
First we prove a Proposition regarding vanishing of relative cohomology. See \cite{berthier1993quelques} for related results.

\

\begin{proposition} \label{vanishing cohomology} 
Let $X$ be a smooth variety,  
and let $\KK \subset \Omega_X$ be a saturated singular integrable Pfaff ideal, with singular set $S = S(\KK)$.
Suppose $\depth_S(\Omega_X/\KK) \ge 2$. Then  
$$\HH^{j}((\Omega_X/\KK)^{h}, d)=0, \ \ \ \forall j > 0.$$  
\end{proposition}

\begin{proof}
We claim that, denoting $U = X-S$,
\begin{equation}
\HH^{j}((\Omega_X/\KK)^{h}|_U, d)=0, \ \ \  \forall j > 0.
\label{vanishing en U}
\end{equation}
Proving this is a local problem for a non-singular Pfaff ideal. 
For each $x \in U$ there exists an analytic open $V$ such that $x \in V \subset U$ and such that $(\Omega_X/\KK)^{h}|_V$ is isomorphic to the relative De Rham complex of a holomorphic submersion $V \to B$ where $B \subset \C^q$ is an open ball. Since such a relative De Rham complex is exact in positive degrees, as shown in  \cite{deligne1969equation} or \cite{MR0508170}, this proves our claim.  

\

\noindent
Now using the hypothesis $\depth_S(\Omega_X/\KK) \ge 2$, by  Proposition \ref{torsion} (f) (as in the proof of \ref{iso}) we have an isomorphism of complexes
$\Omega_X/\KK \to \iota_{*}((\Omega_X/\KK)|_{U})$. Therefore $\HH^{j}((\Omega_X/\KK)^{h}, d) = \iota_{*}(\HH^{j}((\Omega_X/\KK)^h|_{U}), d) = 0$, as we wanted to prove.
\end{proof}

\

\noindent
Next we prove a Lemma  from which we will deduce  $\depth_S(\KK/\JJ) \ge 2$ in the following Corollary. 
See  \cite{MR597077}, Proposition 1.1, for a similar statement, with a different proof.

\

\begin{lemma} \label{depth lemma} 
Let 
\begin{equation}
 0 \to M' \to  M \to M'' \to 0,  \label{projbis}
\end{equation}
be an exact sequence of coherent sheaves on the smooth variety $X$.
If  $M$ is reflexive and $M''$ is torsion-free then $M'$ is reflexive.
\end{lemma}
\begin{proof} Denote $S$ the singular set of $M''$ and let
$p' = \depth_S(M')$, $p = \depth_S(M)$, $p'' = \depth_S(M'')$. By \cite{bourbaki2007elements}, \S 1, Proposition 1, one and only one of the following occurs: 
$$p'= p \le p'', \ \  \  p''=p < p', \ \  \  p'' = p' - 1 < p.$$ 
We have $p \ge 2$ because $M$ is reflexive, and $p'' \ge 1$ because $M''$ is torsion-free. It easily follows that $p' \ge 2$, as claimed.
\end{proof}

\

\begin{corollary} \label{depth} 
Let $\JJ \subset  \KK \subset \Omega_X$ be exterior differential ideals. 
For each $r \in \N$, if $\Omega^r_X/\KK^r$ is torsion-free and
$\Omega^r_X/\JJ^r$ is reflexive then $\KK^r/\JJ^r$ is reflexive.
\end{corollary}
\begin{proof}
Apply Lemma \ref{depth lemma} to the exact sequence 
\begin{equation}
 0 \to \KK^r/\JJ^r \to  \Omega^r_X/\JJ^r \to \Omega^r_X/\KK^r \to 0.  \label{projbis}
\end{equation}
\end{proof}

\

\

\noindent
In the next Proposition \ref{cumbersome}  we obtain the vanishing of the first obstruction, under certain hypothesis, as announced in Remark \ref{strategy}.

\

\begin{proposition} \label{cumbersome}
Let  $X$ be a smooth variety and let $\JJ \subset  \KK \subset \Omega_X$ be saturated singular integrable Pfaff ideals with singular set contained in 
the closed set $S \subset X$.   

\noindent
If $\KK/\JJ$ is reflexive then

(1) $\HH^j((\KK/\JJ)^h), d^h) = 0$ for $j \ge 2$, and

(2) $\HH^{j}(\mathcal{H}omgr_{\Omega_X^{\sharp}}(\JJ, \KK/\JJ)^h, d^h)=0$, for $j \ge 1$.

\noindent
Assuming also $\mathcal{E}xt^1_{O_X}(\JJ^r, \KK^r/\JJ^r) = 0$ for all $r \ge 1$ it follows that

(3) $\mathcal{E}xt^1_{\Omega^h_X}({\JJ}^h, (\KK/\JJ)^h)=0. $
\end{proposition}
\begin{proof}
Denote $\iota: U = X-S \to X$ the inclusion. We know from \ref{vanishing en U} that
$$\HH^{j}((\Omega_X/\JJ)^{h}|_U, d^h)=0  \ \ \text{and} \ \ \HH^{j}((\Omega_X/\KK)^{h}|_U, d^h)=0, \  \forall j \ge 1.$$
Taking cohomology in the exact sequence 
\[
0\to  \iota_*((\KK/\JJ)^h|_U)\to \iota_*((\Omega_X/\JJ)^h|_U)\to \iota_*((\Omega_X/\KK)^h|_U)\to 0,
\]
we obtain $\HH^j(\iota_*((\KK/\JJ)|_U), d^h)=0$ for $j \ge 2$. Since $\KK/\JJ$ is reflexive, 
$\KK/\JJ = \iota_*((\KK/\JJ)|_U)$. Then $(\KK/\JJ)^h = \iota_*((\KK/\JJ)^h|_U)$, and taking homology we obtain (1).

\

\noindent
(2) follows from combining (1) with Proposition \ref{h1}, using again that $\KK/\JJ$ is reflexive. 

\

\noindent
Let us now prove (3). 
From Proposition \ref{h1 y ext1 sheaf} we have
$$\mathcal{E}xt_{\Omega_X}^1({\JJ},{\KK/\JJ}) \cong \HH^1(\mathcal{H}omgr_{\Omega_X^{\sharp}}({\JJ},{\KK/\JJ}),d).$$
Applying $( )^h$ and using (2), we obtain (3).

\end{proof}

\

\begin{corollary}\label{cumbersome2}
Let  $X$ be a smooth variety and let $\JJ \subset  \KK \subset \Omega_X$ be saturated singular integrable Pfaff ideals with singular set contained in 
the closed set $S \subset X$. Suppose:  

i) $\Omega^r_X/\JJ^r$  and $\Omega^r_X/\KK^r$ are reflexive for  $r \ge 2$, 

ii) $\KK^1/\JJ^1$ is reflexive and $\JJ^1$ is locally free, and

iii) $\depth_S \mathcal{H}om_{\OO_X}(\JJ^r, \KK^r/\JJ^r) \ge 3$ for all $r \ge 2$.
 
\noindent
Then  $\mathcal{E}xt^1_{\Omega^h_X}({\JJ}^h, (\KK/\JJ)^h)=0.$
\end{corollary}
\begin{proof}
We only need to check the hypothesis of Proposition  \ref{cumbersome}.

\

\noindent
From Corollary \ref{depth} and i), $\KK^r/\JJ^r$ is reflexive for $r \ge 2$. From ii), also $\KK^1/\JJ^1$ is reflexive. Hence $\KK/\JJ$ is reflexive.

\

\noindent
And from iii) and   Corollary \ref{cor ext=0}, $\mathcal{E}xt^1_{O_X}(\JJ^r, \KK^r/\JJ^r) = 0$ for all $r \ge 2$.
The remaining case $\mathcal{E}xt^1_{O_X}(\JJ^1, \KK^1/\JJ^1) = 0$ follows from $\JJ^1$ being locally free by ii).

\end{proof}

\begin{remark} \label{hypothesis}
The same conclusion $\mathcal{E}xt^1_{\Omega^h_X}({\JJ}^h, (\KK/\JJ)^h)=0$ is obtained, with a similar proof, with the following -perhaps simpler- assumptions:

i') $\Omega^r_X/\JJ^r$  and $\Omega^r_X/\KK^r$ are reflexive for  $r \ge 1$, 

iii') $\depth_S \mathcal{H}om_{\OO_X}(\JJ^r, \KK^r/\JJ^r) \ge 3$ for all $r \ge 1$.
 
\noindent
These hypothesis seem practical for applications, but there are examples of interest where they are not satisfied. For instance, if $X$ is a projective space and
$\JJ$ is the saturated ideal generated by a single homogeneous integrable 1-form then  $\Omega^1_X/\JJ^1$ 
has singularities in codimension two \cite{jouanolou2006equations}
and hence it can not be reflexive. 
\end{remark}

\

\medskip

\subsection{Vanishing of the second obstruction.}   

\
 
\noindent
Our goal in this section is to prove that $H^1(X,\mathcal{H}om_{\Omega_X}(\II_j, \II/\II_j))=0$ under certain conditions. 
We shall do this by reducing to the calculation of a twisted De Rham cohomology group.

\medskip
\noindent
Let $C$ be a differential graded $\Omega_X$-Module. Then (see e. g.  \cite{MR0102537}, \cite{MR0217085}, \cite{griffiths2014principles})
we have two spectral sequences, $I(C)$ and ${II}(C)$, both converging
to the hypercohomology vector spaces $\mathbb{H}^n(C)$, and whose initial terms are:
\[
I_2^{p,q}(C) = H^p(X, \HH^q(C,d)), \quad \quad
II_2^{p,q}(C) = H^q( H^p(X, C) , d).
\]

\noindent
Also, we have the two similar spectral sequences $I(C^h)$ and $II(C^h)$ obtained from the differential graded $\Omega^h_X$-Module $C^h$ on the analytic variety $X^h$, as in Definition \ref{analytic sheaf}. Both converge to $\mathbb{H}^n(C^h)$ and their initial terms are:
\[
I_2^{p,q}(C^h) = H^p(X, \HH^q(C^h,d)), \quad \quad
II_2^{p,q}(C^h) = H^q( H^p(X, C^h) , d).
\]

\

\noindent
Recall also that, if $C^j = 0$ for $j<0$, we have the exact sequence of low degree terms (see e. g. \cite{griffiths2014principles}, p. 458), written for instance for $I(C)$:
\begin{equation}
0 \to I_2^{1,0}(C) \to \mathbb{H}^1(C) \to I_2^{0,1}(C) \to I_2^{2,0}(C) \to \mathbb{H}^2(C).
\label{low degree terms}
\end{equation}

\medskip
\noindent
Let $\JJ \subset \KK$ be singular integrable Pfaff ideals with singular set $S=S(\KK)$. As in \ref{homgr2}, we take 
\begin{equation}
C = \mathcal{H}omgr_{\Omega_X^{\sharp}}(\JJ, G),
\label{defC}
\end{equation}
with $G=\KK/\JJ$. Assume $\depth_S(G) \ge 2$.  

\medskip
\noindent
Let us consider the first spectral sequence $I(C^h)$. 
By Proposition \ref{cumbersome} (2) we know $\HH^q(C^h, d^h) = 0$ for $q \ge 1.$  
Therefore, $I_2^{p,q}(C^h) = 0$  for $q \ge 1$ (the spectral sequence degenerates at the second page) and
\begin{equation}
\mathbb{H}^p(C) = \mathbb{H}^p(C^h) = H^p(X, \HH^0(C^h,d)) = H^p(X, \mathcal{H}om_{\Omega_X}(\JJ, G)^h), \ \forall p \ge 0.
\label{hyper1}
\end{equation}
In particular, 
\begin{equation}
H^1(X, \mathcal{H}om_{\Omega_X}(\JJ, \KK/\JJ)^h) = \mathbb{H}^1(C).
\label{hyper2}
\end{equation}

\

\noindent
Now let us compute $\mathbb{H}^1(C)$ using the second spectral sequence.

\medskip
\noindent
Assume that $K \subsetneq \Omega_X$ is a proper exterior ideal with $K^0 = 0$. Therefore $G^j=0$ for $j \le 0$, and hence, $C^j = 0$ for $j<0$. By \ref{low degree terms}, $\mathbb{H}^1(C) = 0$  follows if we show

\begin{equation}
II_2^{1,0}(C) = H^0( H^1(X, C) , d) = 0,
\label{a}
\end{equation}
and 
\begin{equation}
II_2^{0,1}(C) = H^1( H^0(X, C) , d) =  0.
\label{b}
\end{equation}

\medskip
\noindent
Regarding \ref{a}, we will make the hypothesis that
\begin{equation}
H^1(d^0): H^1(X, C^0) \to H^1(X, C^1),
\label{hyper3}
\end{equation}
is injective. This holds, for instance, if
\begin{equation}
H^1(X, C^0) = H^1(X, \mathcal{H}om_{\Omega_X^{\sharp}}(\JJ, \KK/\JJ)) = H^1(X, \mathcal{H}om_{\OO_X}({\JJ}^1, {\KK}^1/{\JJ}^1)) = 0.
\label{hyper4}
\end{equation}
It is clear that  \ref{a} is equivalent to  the injectivity in \ref{hyper3}, since $H^0( H^1(X, C) , d) = \ker H^1(d^0)$, by definition. On the other hand, 
the condition on coherent sheaves 

\noindent
$H^1(X, \mathcal{H}om_{\OO_X}({\JJ}^1, {\KK}^1/{\JJ}^1)) = 0$ of \ref{hyper4}  holds in many examples.

\

\medskip
\noindent
In the next Proposition we show that \ref{b} is true in case $X$ is a complex projective space.

\medskip
\noindent
For convenience, let us recall  some notations and well-known facts to be used below.

\

\noindent
Let $\pi: U = \C^{n+1}-\{0\} \to \PP^n(\C) = \PP^n$ be the canonical projection and denote $A = \C[x_0, \dots, x_n]$. 
If $M$ is an $A$-module here we denote $\bar M$ the associated quasi-coherent sheaf on $\C^{n+1}$.
Considering $A$ as a graded $\C$-algebra, if $M = \bigoplus_{r \in \Z} M_r$ is a graded $A$-module  denote as usual $\tilde{M}$
the associated quasi-coherent sheaf on $\PP^n$.   
If $F$ is a coherent sheaf on $\PP^n$ then $M = \Gamma_*(F) = \bigoplus_{r \in \Z} H^0(\PP^n, F(r))$ is a graded $A$-module, and $F \cong \tilde M$.
It is not hard to prove from \cite{MR0463157} (II, Proposition (5.12) (c)) that for a graded $A$-module $M$, $\pi^*(\tilde M) =  \bar M|_U$. 
In particular, $\pi^*(\tilde M)$ extends to $\C^{n+1}$; see also \cite{MR0463157} (Ex. II (5.15)) for more information on extending coherent sheaves. 

\

\noindent
If $M$ and $N$ are graded $A$-modules, and $M$ is finitely generated, then $\Hom_A(M, N)$ is a graded $A$-module
with homogeneous component of degree $r \in \Z$ defined as:
\begin{equation}
\Hom_A(M, N)_r = \{ \varphi \in \Hom_A(M, N)/ \varphi(M_s) \subset N_{s+r}, \forall s \in \Z\},
\label{graded}
\end{equation}
see \cite{bourbakialgebra}, \S11(6). 

\

\noindent
For $F$, $G$ coherent sheaves on $\PP^n$,  the graded module 
$H = \Gamma_* \mathcal{H}om_{\OO_{\PP^n}}(F, G)$ has homogeneous component of degree $r$ equal to
\begin{equation}
H_r=H^0(\PP^n, \mathcal{H}om_{\OO_{\PP^n}}(F, G)(r)) = \Hom_{\OO_{\PP^n}}(F, G(r)). 
\label{graded2}
\end{equation}
Since
$\pi^* \mathcal{H}om_{\OO_{\PP^n}}(F, G) = \mathcal{H}om_{\OO_{\C^{n+1}}}(\pi^* F, \pi^* G)|_U$ (see \cite{serre1956geometrie}, Proposition 21)
we obtain  
\begin{equation}
\mathcal{H}om_{\OO_{\C^{n+1}}}(\pi^* F, \pi^* G)= \bar H.
\label{graded3}
\end{equation}

\

\

\noindent
If $V \subset \C^{n+1}$ is open and $\omega$ is a differential form on $V - \{0\}$
then $\omega$ extends uniquely to  $V$. Then we have
$\iota_* (\Omega_{\C^{n+1} -\{0\}}) = \Omega_{\C^{n+1}}$, where $\iota: \C^{n+1} -\{0\} \to \C^{n+1}$ is the inclusion.

\

\noindent
It is easy to check that the pull-back of a 1-form on some open set $U \subset \PP^n$
is a homogeneous rational 1-form of degree zero on $\pi^{-1}(U).$
This means that the image of the natural pull-back of 1-forms 
$\pi^*: \Omega^1_{\PP^n}   \to  \pi_* \Omega^1_{\C^{n+1}}$ 
is contained in the subsheaf $\tilde {E^1} \subset \pi_* \Omega^1_{\C^{n+1}}$, 
where $E^1 = H^0(\C^{n+1}, {\Omega^1_{\C^{n+1}}})$ is the graded $A$-module of polynomial 1-forms on $\C^{n+1}$.
One thus obtains the Euler exact sequence:
\begin{equation}
\xymatrix{
0 \ar[r] & \Omega^1_{\PP^n}  \ar[r]^{\pi^*} &  \tilde{E^1} \ar[r]^{\iota_R} &  \OO_{\PP^n} \ar[r] & 0,
 }
\label{euler} 
\end{equation}
where $\iota_R$ is contraction with the radial vector field $R = \sum_{0 \le i \le n} x_i \frac{\partial}{\partial x_i}$.

\noindent
Taking exterior powers in \ref{euler} (\cite{MR0463157}, II, Ex. 5.16, see also  \cite{okonek1980vector} p. 3, \cite{hirzebruch1966topological} p. 55)  
we get the exact sequences for $q \le n$ (Euler sequence for $q$-forms):
\begin{equation}
\xymatrix{
0 \ar[r] & \Omega^q_{\PP^n}  \ar[r]^{\pi^*} & \bigwedge^q \tilde{E^1}  \ar[r]^{\iota_R} &  \Omega^{q-1}_{\PP^n} \ar[r] & 0.
 }
\label{euler2} 
\end{equation}
 \noindent
 Since $\Omega^1_{\C^{n+1}}$ is free, with basis elements $dx_0, \dots, dx_n$,
 $\bigwedge^q \Omega^1_{\C^{n+1}} := \Omega^q_{\C^{n+1}}$ is also free, with basis elements $dx_{i_1} \wedge \dots \wedge dx_{i_q}$ for 
 all choices of $0 \le i_1 < \dots < i_q \le n$ .
It follows that the natural map 
 $\bigwedge^q {E^1} = \bigwedge^q H^0(\C^{n+1}, {\Omega^1_{\C^{n+1}}}) \to E^q := H^0(\C^{n+1}, {\Omega^q_{\C^{n+1}}})$ is an isomorphism
of graded modules, and  hence we may replace $\bigwedge^q \tilde{E^1} = \tilde  {E^q}$ in \ref{euler2}.

\
 
\noindent
Because the mentioned basis elements of $E^1$ (resp. of $E^q$) are homogeneous of degree one (resp. of degree $q$), we have natural isomorphisms of graded $A$-modules
 ${E^1}  =  A(-1)^{n+1}$ and ${E^q}  =  \bigwedge^q (\C^{n+1}) \otimes_{\C} A(-q) = A(-q)^{\binom{n+1}{q}}$; then we may replace
 $ \tilde{E^1}  =  \OO_{\PP^n}(-1)^{n+1}$ in \ref{euler}, and 
 $\bigwedge^q \tilde{E^1} =  \OO_{\PP^n}(-q)^{\binom{n+1}{q}}$  in \ref{euler2}.
 
 \
 
\noindent
For each $r \in \N$, tensoring \ref{euler2} by $\OO_{\PP^n}(r)$ and taking global sections we obtain the exact sequence
\begin{equation}
\xymatrix{
\small
0 \ar[r] & H^0(\PP^n, \Omega^q_{\PP^n}(r))  \ar[r]^{\pi^*} & H^0(\C^{n+1}, {\Omega^q_{\C^{n+1}}})_r  \ar[r]^{\iota_R} &  H^0(\PP^n, \Omega^{q-1}_{\PP^n}(r)) \ar[r] & 0.
 }
\label{euler3} 
\end{equation}
For the surjectivity of  $\iota_R$ see \cite{okonek1980vector} p. 4.
 
 \

\begin {proposition} \label{vanishing}
Let $X = \PP^n(\C) = \PP^n$ and let $\JJ \subset \KK \subsetneq \Omega_X$ be integrable singular Pfaff ideals with $K^0 = 0$. 
Denote $S \subset \C^{n+1}$ the singular set of $\pi^* \KK$.  As in \ref{defC} define $C = \mathcal{H}omgr_{\Omega_X^{\sharp}}(\JJ, \KK/\JJ).$ 
If $\pi^*(\KK/\JJ)$ is reflexive then
$$H^1( H^0(X, C) , d) = 0.$$
\end {proposition}  

\begin{proof}
The proof will have some points in common with the proof of Proposition \ref{h1}.
The claim amounts to the exactness of the short complex of finite dimensional complex vector spaces:
\begin{equation}
\xymatrix{
\Hom_{\Omega_X^{\sharp}}(\JJ,G) \ar[r]^{d^{0}} & \Hom_{\Omega_X^{\sharp}}(\JJ,G[1])  \ar[r]^{d^1} & 
\Hom_{\Omega_X^{\sharp}}(\JJ,G[2]),
}
\label{short complex3} 
\end{equation}
where $G = \KK/\JJ$. By Remark \ref{complex O} it suffices to prove the exactness of the isomorphic complex
\begin{equation}
\xymatrix{
\Hom_{\OO_X}(\JJ^1,G^{1}) \ar[r]^{d_{\OO}^{0}} & \Hom_{\OO_X}(\JJ^1,G^{2})  \ar[r]^{d_{\OO}^1} &\Hom_{\OO_X}(\JJ^1,G^{3}).
 }
\label{short complex4} 
\end{equation}

\

\noindent
From Proposition \ref{cumbersome} on $\C^{n+1}$,  with the ideals
$\pi^* \JJ \subset \pi^* \KK \subset \Omega_{\C^{n+1}}$, we know that the following sequence of analytic sheaves is exact:
\begin{equation}
\xymatrix{
 \mathcal{H}om_{\OO_{\C^{n+1}}}(\pi^*\JJ^1, \pi^*G^1)^h  \ar[r]^{d_{\OO}^{0}} &  \mathcal{H}om_{\OO_{\C^{n+1}}}(\pi^*\JJ^1, \pi^*G^2)^h  \ar[r]^{d_{\OO}^1} &   
 \mathcal{H}om_{\OO_{\C^{n+1}}}(\pi^*\JJ^1, \pi^*G^3)^h.
 }
\label{short complex7} 
\end{equation}

\

\noindent
Let $\eta \in \Hom_{\OO_X}(\JJ^1,G^{2})$ be such that $d_{\OO}^1(\eta) = 0$. 
We want to show that 
\begin{equation}
\eta = d_{\OO}^0(\theta),
\label{boundary} 
\end{equation}
for some $\theta \in \Hom_{\OO_X}(\JJ^1,G^{1})$.

\

\noindent
Denote ${\eta}_0  \in \mathcal{H}om_{\OO_{\C^{n+1}}}(\pi^*\JJ^1, \pi^*G^2)^h_0$  the analytic germ at $0$ of $\pi^*{\eta}$. 
Taking germs at $0$ in \ref{short complex7}, there exists   ${\theta}_0 \in \mathcal{H}om_{\OO_{\C^{n+1}}}(\pi^*\JJ^1, \pi^*G^1)^h_0$
such that 
\begin{equation}
 {\eta}_0 = d_{\OO}^{0}({\theta}_0).
\label{boundary2} 
\end{equation} 

\

\noindent
Consider the Taylor expansion of ${\theta}_0$ around the origin. It may be written as
\begin{equation}
{\theta}_0 = \sum_{r \ge 0} {\theta}_0(r),
\label{theta}
\end{equation}
where ${\theta}_0(r) \in \Hom_{\OO_{\C^{n+1}}}(\pi^*\JJ^1, \pi^* G^1)_r =  \pi^* \Hom_{\OO_{X}}(\JJ^1, G^1(r))$ is a   global section, 
 homogeneous of degree $r$, for $r \ge 0$. We denote 
 \begin{equation}
{\theta}_0(r) = \pi^* \theta(r),
\label{theta2}
\end{equation}
for $\theta(r) \in \Hom_{\OO_{X}}(\JJ^1, G^1(r)).$

\

\noindent
To see \ref{theta},  let us take a closer look into    $\mathcal{H}om_{\OO_{\C^{n+1}}}(\pi^*\JJ^1, \pi^*G^1)^h_0$.
First, if $M$ is an $A$-module, consider $\bar M^h_0$, the germs at $0$ of the analytification of the Zariski sheaf $\bar M$.
By its definition, see \cite{serre1956geometrie}, a typical element $\mu  \in \bar M^h_0$ is represented by a finite sum 
$\mu = \sum_{1 \le i \le N} \frac {a_i}{f_i} \ m_i$
where the $a_i$ are holomorphic complex valued functions defined around $0$, $f_i \in A$ are such that $f_i(0) \ne 0$, and $m_i \in M$.
Let us write the Taylor expansion  $\frac {a_i}{f_i} = \sum_{r=0}^{\infty} b_i(r)$, where $b_i(r) \in A_r$ is a homogeneous polynomial of degree $r$.
We obtain $\mu = \sum_{1 \le i \le N}  \sum_{r=0}^{\infty} b_i(r) \ m_i =  \sum_{r=0}^{\infty} \mu(r)$ with $\mu(r) = \sum_{1 \le i \le N}  b_i(r) m_i$.
If $M$ is a graded $A$-module we may assume that each $m_i$ is homogeneous. Then each $b_i(r) m_i \in M$ is homogeneous, and
collecting homogeneous terms of the same degree, and changing the notation, we obtain
$\mu =   \sum_{r=0}^{\infty} \mu(r)$ with $\mu(r) \in M_r$. Applying this to the graded $A$-module $H$ as in \ref{graded3}, we obtain \ref{theta}.

\

\

\noindent
Applying $d_{\OO}^{0}$ in \ref{theta} we obtain
\begin{equation}
 {\eta}_0 = \pi^* \eta = d_{\OO}^{0}({\theta}_0)  = \sum_{r \ge 0} d_{\OO}^{0}({\theta}_0(r)).
\label{igualdad}
\end{equation}

\noindent
By \ref{graded2}, $\pi^* \eta$ has degree zero. 

\

\noindent
On the other hand, $d_{\OO}^{0}$  is homogeneous of degree zero. This follows easily from Remark \ref{degree} and Example \ref{formula d0}, \ref{differential4}.
Another way to see it: in \ref{differential3}, since de De Rham differential in $\C^{n+1}$ is homogeneous of degree zero,
$d_G$ and $d_{\II}$ are homogeneous of degree zero. 
Hence, if $\varphi$ is homogeneous of degree $r$ then, for any $k \ge 0$, $d_{\OO}^k(\varphi^1)$ is also homogeneous of the same degree $r$.
Therefore $d_{\OO}^{0}({\theta}_0(r))$ is homogeneous of degree $r$.

\

\noindent
Looking at the homogeneous component of degree zero in \ref{igualdad} we get
$\pi^* \eta  = d_{\OO}^{0}({\theta}_0(0)).$
Then, from \ref{theta2} we obtain $\eta = d_{\OO}^{0}({\theta}(0))$, which proves \ref{boundary}.

\

\noindent
Let us remark that w e showed that the differential in the complex  \ref{short complex7}  preserves degree.
Hence the exactness of   \ref{short complex7}  implies the exactness of its homogeneous component of degree zero,
which is \ref{short complex4}. 
\end{proof}

\

\subsection{Second theorem on stability.}  

\

\

\noindent
Summarizing, we obtain:

\

\begin{theorem}[B] \label{theorem pfaff2}
Denote $X = \PP^n(\C)$ and $\pi: \C^{n+1}-\{0\} \to X$  the canonical projection.
Fix $q$ such that $2 \le q \le n$ and for $j= 1, \dots, q$  let $\mathcal{F}_j$ be an irreducible component of $\Hilb(\Omega_X)$. For $\II_j \in \mathcal{F}_j$
denote $\II = \sum_{j=1}^q \II_j$, and $S = S(\II)$. We assume that for all $j$ and general $\II_j \in \mathcal{F}_j$
the following conditions are satisfied:

\
 
\noindent
1) $\II_j$ is a saturated singular Pfaff ideal, 

\noindent
2) $\II$ is saturated,

\noindent
3) $\Hilb(\Omega_X)$ is reduced at a general point of  $\mathcal{F}_j,$

\noindent
4) $\mathcal{E}xt^1_{O_X}(\II_j^r, \II^r/\II_j^r) = 0$ for all $r \ge 1$,

\noindent
5) $\II/\II_j$ is reflexive, 

\noindent
6)  $\pi^*(\II/\II_j)$ is reflexive at $0$,   and

\noindent
7) $H^1(X, \mathcal{H}om_{\OO_X}(\II_j^1, \II^1/\II_j^1)) = 0.$

\

\noindent
Then the derivative of $\mu_{\tilde \II}$  (\ref{sigma}) at a general point is surjective. 
Therefore, $\langle\mathcal{F}_1, \dots, \mathcal{F}_q\rangle$ is an irreducible component of $\Hilb(\Omega_X)$. Also, $\Hilb(\Omega_X)$ is reduced at a general point of $\langle\mathcal{F}_1, \dots, \mathcal{F}_q\rangle$.
\end{theorem}
\begin{proof} 
We only need to verify the validity of hypotheses a) and b) of Theorem (A)  \ref{theorem pfaff}. 

\noindent
a) is our current 3). 

\noindent
As we saw in Remarks \ref{remarkpfaff} and \ref{strategy}, b) follows from 
$$
H^0(X,\mathcal{E}xt_{\Omega_X}^{1}({\II_j},{\II/\II_j})^h) = 0  \ \ \text{and} \ \ H^1(X,\mathcal{H}om_{\Omega_X}({\II_j},{\II/\II_j})^h) = 0, \ \forall j.
$$

\noindent
From 4), 5) and Proposition  \ref{cumbersome} we obtain the first condition. 
The second condition follows from \ref{b} and Proposition \ref{vanishing}, because of 5) and 6).
Notice that we are using \ref{hyper4}, which is true due to 7).   
Therefore the proof is complete.
\end{proof}

\medskip

\newpage

\section {Construction of Examples.}  \label{examples}

\

\noindent
In order to give examples of irreducible components of type $\langle\mathcal{F}_1, \dots, \mathcal{F}_q\rangle$,  
one needs to verify the validity of the hypotheses of Theorem (B)  \ref{theorem pfaff2}. 

\

\noindent
Let us make some comments about these hypotheses.

 \

\noindent
1) The components $\mathcal{F}_j$ are chosen so that their general members $\II_j$ are the saturation of singular Pfaff ideals.

\

\noindent
2) The fact that $\II$ is saturated for general $\II_j$'s does not seem automatic and should be verified separately in examples.

\

\noindent
3) The participating irreducible components $\mathcal{F}_j$ of $\Hilb(\Omega_X)$ 
are chosen so that $\Hilb(\Omega_X)$ is reduced at a general point of  $\mathcal{F}_j$.
Several examples of such generically reduced components are known: foliations with split tangent sheaf \cite{cukierman2008stability}, rational foliations of codimension $q$ \cite{cukierman2009stability}, logarithmic foliations of codimension one \cite{cukierman2019stability}, 
some foliations of codimension one and degree three \cite{costa2022codimension}, certain pull-back foliations \cite{gargiulo}, and some $S$-logarithmic foliations \cite{Chehebar}.

\

\noindent
4) To prove the required vanishing of $\mathcal{E}xt^1_{O_X}(\II_j^r, \II^r/\II_j^r)$ one option is to use Theorem \ref{ext=0};
see Corollary \ref{cumbersome2} and Remark \ref{hypothesis}. This reduces the issue to checking certain $\depth$ conditions, as in Corollary \ref{cor ext=0}.
More precisely, for each $r$ it would be enough to show that
$$\depth_S \II^r/\II_j^r \ge 2  \  \text{and} \  \depth_S \mathcal{H}om_{\OO_X}(\II_j^r, \II^r/\II_j^r) \ge 3, \ \forall j.$$
Calculating these $\depth$s seems non-trivial, and we plan to carry out some such calculations in a separate article.

\

\noindent
5)  For each $r \in \N$, by Corollary \ref{depth} we know that $\II^r/\II^r_j$ is reflexive 
if $\Omega^r_X/\II^r$ is torsion-free (which is hypothesis 2) again) and $\Omega^r_X/\II^r_j$ is reflexive.
As in 4), one needs to prove that $\depth_S \Omega^r_X/\II^r_j \ge 2$.

\

\noindent
6) Similarly to 5), for each degree $r$, this hypothesis follows from the condition that  $\pi^*(\Omega^r_X/\II^r_j)$ is reflexive at $0$.

\

\noindent
7) This seems reasonably easy to analyze in specific situations. For instance, if $X$ is a projective space, each $\II_k^1$ is graded-free, and 
the sum  $\II^1 = \sum_k \II_k^1$ is a direct sum, then $\II^1/\II_j^1 = \bigoplus_{k \ne j} \II^1_k$ and hence
$H^1(X, \mathcal{H}om_{\OO_X}(\II_j^1, \II^1/\II_j^1)) =  \bigoplus_{k \ne j} H^1(X, \II^1_k \otimes (\II^1_j)^*) = 0.$

\newpage

\bibliographystyle{plain}
\bibliography{citas}

\begin{thebibliography}{10}

\bibitem{MR1863391}
M.~Artin and J.~J. Zhang.
\newblock Abstract {H}ilbert schemes.
\newblock {\em Algebras and Representation Theory}, 4(4):305--394, 2001.

\bibitem{berthier1993quelques}
M.~Berthier and D.~Cerveau.
\newblock Quelques calculs de cohomologie relative.
\newblock In {\em Annales scientifiques de l'Ecole normale sup{\'e}rieure},
  volume~26, pages 403--424, 1993.

\bibitem{bourbaki2007algebre}
N.~Bourbaki.
\newblock {\em Alg{\`e}bre commutative: Chapitres 5 {\`a} 7}.
\newblock Springer Science \& Business Media, 2007.

\bibitem{bourbakihomologicalalgebra}
N.~Bourbaki.
\newblock Elements of mathematics. {A}lgebra. {C}hapter 10.
\newblock 2007.

\bibitem{bourbakialgebra}
N.~Bourbaki.
\newblock Elements of mathematics. {A}lgebra {I}. {C}hapter 2.
\newblock 2007.

\bibitem{bourbaki2007elements}
N.~Bourbaki.
\newblock Elements of mathematics. {C}ommutative algebra. {C}hapter 10.
\newblock 2007.

\bibitem{MR1083148}
R.~L. Bryant, S.~S. Chern, R.~B. Gardner, H.~L. Goldschmidt, and P.~A.
  Griffiths.
\newblock {\em Exterior differential systems}, volume~18 of {\em Mathematical
  Sciences Research Institute Publications}.
\newblock Springer-Verlag, New York, 1991.

\bibitem{bryant1995characteristic}
R.~L. Bryant and P.~A. Griffiths.
\newblock Characteristic cohomology of differential systems: General theory.
\newblock {\em Journal of the American Mathematical Society}, 8(3):507--596,
  1995.

\bibitem{Chehebar}
M.~Chehebar.
\newblock Thesis.
\newblock {\em Componentes S-logaritmicas del espacio de moduli de foliaciones,
  \url{https://web.dm.uba.ar/index.php/docencia/tesis-de-doctorado}}, 2023.

\bibitem{cukierman2019stability}
F.~Cukierman, J.~Gargiulo~Acea, and C.~Massri.
\newblock Stability of logarithmic differential one-forms.
\newblock {\em Transactions of the American Mathematical Society}, 2019.

\bibitem{cukierman2008stability}
F.~Cukierman and J.~V. Pereira.
\newblock Stability of holomorphic foliations with split tangent sheaf.
\newblock {\em American Journal of Mathematics}, 130(2):413--439, 2008.

\bibitem{cukierman2009stability}
F.~Cukierman, J.~V. Pereira, and I.~Vainsencher.
\newblock Stability of foliations induced by rational maps.
\newblock In {\em Annales de la Facult{\'e} des sciences de Toulouse:
  Math{\'e}matiques}, volume~18, pages 685--715, 2009.

\bibitem{costa2022codimension}
R.~C. da~Costa, R.~Lizarbe, and J.~V. Pereira.
\newblock Codimension one foliations of degree three on projective spaces.
\newblock {\em Bulletin des Sciences Math{\'e}matiques}, 174:103092, 2022.

\bibitem{deligne1969equation}
P.~Deligne.
\newblock Equations diff{\'e}rentielles a points singulieres r{\'e}gulieres.
\newblock {\em Lecture Notes in Mathematics}, 163, 1969.

\bibitem{MR1322960}
D.~Eisenbud.
\newblock {\em Commutative algebra}.
\newblock Graduate Texts in Mathematics. Springer-Verlag, 1995.

\bibitem{MR2222646}
B.~Fantechi, L.~G{\"o}ttsche, L.~Illusie, S.~L. Kleiman, N.~Nitsure, and
  A.~Vistoli.
\newblock {\em Fundamental algebraic geometry}.
\newblock Mathematical Surveys and Monographs. American Mathematical Society,
  Providence, RI, 2005.

\bibitem{MR2320462}
L.~Fiorot.
\newblock On derived categories of differential complexes.
\newblock {\em J. Algebra}, 312(1):362--376, 2007.

\bibitem{gargiulo}
J.~Gargiulo~Acea, A.~Molinuevo, F.~Quallbrunn, and S.~Velazquez.
\newblock Stability of pullbacks of foliations on weighted projective spaces.
\newblock {\em \url{https://arxiv.org/abs/2212.12974}}.

\bibitem{MR2394437}
I.~M. Gelfand, M.~M. Kapranov, and A.~V. Zelevinsky.
\newblock {\em Discriminants, resultants and multidimensional determinants}.
\newblock Modern Birkh\"auser Classics. Birkh\"auser, 2008.

\bibitem{gomez1988transverse}
X.~Gomez-Mont.
\newblock The transverse dynamics of a holomorphic flow.
\newblock {\em Annals of Mathematics}, 127(1):49--92, 1988.

\bibitem{griffiths2014principles}
P.~Griffiths and J.~Harris.
\newblock {\em Principles of algebraic geometry}.
\newblock John Wiley \& Sons, 2014.

\bibitem{grothendieck1956faisceaux}
A~Grothendieck.
\newblock Sur les faisceaux alg{\'e}briques et les faisceaux analytiques
  coh{\'e}rents.
\newblock {\em S{\'e}minaire Henri Cartan}, 9:1--16, 1956.

\bibitem{MR0102537}
A.~Grothendieck.
\newblock Sur quelques points d'alg\`ebre homologique.
\newblock {\em T\^ohoku Math. J. (2)}, 9:119--221, 1957.

\bibitem{MR0217083}
A.~Grothendieck.
\newblock El\'ements de g\'eom\'etrie alg\'ebrique {I}. {L}e langage des
  sch\'emas.
\newblock {\em Inst. Hautes \'Etudes Sci. Publ. Math.}, (4):228, 1960.

\bibitem{MR0217085}
A.~Grothendieck.
\newblock \'{E}l\'ements de g\'eom\'etrie alg\'ebrique {III}. \'{E}tude
  cohomologique des faisceaux coh\'erents. {I}.
\newblock {\em Inst. Hautes \'Etudes Sci. Publ. Math.}, (11):167, 1961.

\bibitem{grothendieck2005cohomologie}
A.~Grothendieck and M.~Raynaud.
\newblock Cohomologie locale des faisceaux coh\'erents et th\'eoremes de
  {L}efschetz locaux et globaux ({SGA 2}).
\newblock {\em arXiv math/0511279}, 2005.

\bibitem{MR0224620}
R.~Hartshorne.
\newblock {\em Local cohomology}.
\newblock Springer-Verlag, 1967.

\bibitem{MR0463157}
R.~Hartshorne.
\newblock {\em Algebraic geometry}.
\newblock Springer-Verlag, 1977.

\bibitem{MR597077}
R.~Hartshorne.
\newblock Stable reflexive sheaves.
\newblock {\em Math. Ann.}, 254(2):121--176, 1980.

\bibitem{MR2583634}
R.~Hartshorne.
\newblock {\em Deformation theory}.
\newblock Graduate Texts in Mathematics. Springer, 2010.

\bibitem{MR0310287}
M.~Herrera and D.~Lieberman.
\newblock Duality and the de {R}ham cohomology of infinitesimal neighborhoods.
\newblock {\em Invent. Math.}, 13:97--124, 1971.

\bibitem{MR0346025}
P.~Hilton and U.~Stammbach.
\newblock {\em A course in homological algebra}.
\newblock Springer-Verlag, 1971.

\bibitem{hirzebruch1966topological}
F.~Hirzebruch, A.~Borel, and R.~Schwarzenberger.
\newblock {\em Topological methods in algebraic geometry}, volume 175.
\newblock Springer Berlin-Heidelberg-New York, 1966.

\bibitem{jinnah1975reflexive}
M.~I. Jinnah.
\newblock Reflexive modules over regular local rings.
\newblock {\em Archiv der Mathematik}, 26:367--371, 1975.

\bibitem{jothilingam2016projectivity}
P.~Jothilingam and T.~Duraivel.
\newblock Projectivity of reflexive modules over regular rings.
\newblock {\em Journal of Algebra and Its Applications}, 15(2), 2016.

\bibitem{jouanolou2006equations}
J.-P. Jouanolou.
\newblock {\em Equations de Pfaff alg{\'e}briques}, volume 708.
\newblock Springer, 2006.

\bibitem{MR1181207}
M.~Kapranov.
\newblock On {DG}-modules over the de {R}ham complex and the vanishing cycles
  functor.
\newblock In {\em Algebraic geometry ({C}hicago, {IL}, 1989)}, volume 1479 of
  {\em Lecture Notes in Math.}, pages 57--86. Springer, 1991.

\bibitem{MR1258406}
B.~Keller.
\newblock Deriving {DG} categories.
\newblock {\em Ann. Sci. \'Ecole Norm. Sup. (4)}, 27(1):63--102, 1994.

\bibitem{kleiman1971geometry}
S.~L. Kleiman and J.~Landolfi.
\newblock Geometry and deformation of special schubert varieties.
\newblock {\em Compositio Mathematica}, 23(4):407--434, 1971.

\bibitem{lang2004algebra}
S.~Lang.
\newblock Algebra.
\newblock {\em Graduate Texts in Mathematics}, 2004.

\bibitem{MR0508170}
B.~Malgrange.
\newblock Frobenius avec singularit\'es {II}. {L}e cas g\'en\'eral.
\newblock {\em Invent. Math.}, 39(1):67--89, 1977.

\bibitem{matsumura1989commutative}
H.~Matsumura.
\newblock {\em Commutative ring theory}.
\newblock Number~8. Cambridge University Press, 1989.

\bibitem{mumford1966lectures}
D.~Mumford and G.~Bergman.
\newblock {\em Lectures on curves on an algebraic surface}.
\newblock Princeton University Press, 1966.

\bibitem{MR3640821}
S.~Nasseh and S.~Sather-Wagstaff.
\newblock Extension groups for {DG} modules.
\newblock {\em Comm. Algebra}, 45(10):4466--4476, 2017.

\bibitem{okonek1980vector}
C.~Okonek, M.~Schneider, H.~Spindler, and S.~Gelfand.
\newblock {\em Vector bundles on complex projective spaces}.
\newblock Springer, 1980.

\bibitem{MR3386225}
S.~Rybakov.
\newblock D{G}-modules over de {R}ham {DG}-algebra.
\newblock {\em Eur. J. Math.}, 1(1):25--53, 2015.

\bibitem{schlessingerthesis}
M.~Schlessinger.
\newblock Thesis.
\newblock {\em Infinitesimal Deformations of Singularities}, 1964.

\bibitem{sernesi2007deformations}
E.~Sernesi.
\newblock {\em Deformations of algebraic schemes}.
\newblock Springer, 2007.

\bibitem{serre1956geometrie}
J.-P. Serre.
\newblock G{\'e}om{\'e}trie alg{\'e}brique et g{\'e}om{\'e}trie analytique.
\newblock In {\em Annales de l'institut Fourier}, volume~6, pages 1--42, 1956.

\bibitem{stacks-project}
The {Stacks project authors}.
\newblock The stacks project.
\newblock \url{https://stacks.math.columbia.edu}, 2025.

\bibitem{toen2017problemes}
B.~To{\"e}n.
\newblock Problemes de modules formels.
\newblock {\em Seminaire Bourbaki}, 2017.

\end{thebibliography}

\newpage

\begin{flushleft}
\begin{small}

\noindent
 Fernando Cukierman, \newline
 Universidad de Buenos Aires and CONICET, \newline
 fcukier{@}dm.uba.ar
 
 \medskip

\noindent
C\'esar Massri, \newline
Universidad de Buenos Aires and CONICET, \newline
cmassri{@}dm.uba.ar

\medskip
Address: \newline
 Departamento de Matem\'{a}tica, FCEN. \newline
 Ciudad Universitaria.\newline
 (1428) Buenos Aires.\newline
 ARGENTINA.

\end{small}
\end{flushleft}

\end{document}